\documentclass[11pt]{amsart}

\usepackage{url}
\usepackage{latexsym}

\usepackage{amsmath,amssymb,amsthm,amsfonts,latexsym,enumerate}

\usepackage{graphics}
\usepackage[active]{srcltx}

\usepackage{color}
\usepackage[usenames,dvipsnames,svgnames,table]{xcolor}
\usepackage[colorlinks=false,
            linkcolor=purple,
            urlcolor=blue,
            citecolor=gray]{hyperref}
\usepackage{pstricks,pst-plot,pst-node,pst-text,pst-3d}

\usepackage{srcltx}

\let\=\partial
\thanks{Research partially supported by the NCN grant
2014/13/B/ST1/00543.
}

\newtheorem {pro}{Proposition}[section]
\newtheorem {thm}[pro]{Theorem}
\newtheorem {cor}[pro]{Corollary}
\newtheorem{lem}[pro]{Lemma}
\newtheorem {prodfn}[pro]{Proposition and Definition}

\theoremstyle{definition}
 \newtheorem{rem}[pro]{Remark}

\newtheorem{dfn}[pro]{Definition}

\theoremstyle{definition}
 \newcommand{\Q} {\mathbb{Q}}

\newcommand{\bt} {\mathbf{t}}

\newcommand{\f} {\mathcal{F}}
\newcommand{\pa}{\partial}
\newcommand{\ws}{\overline{\Omega}}
\newcommand{\etab}{\overline{\eta}}

\newcommand{\wsp}{\overline{\Omega}_{p}}
\newcommand{\hp}{H_{p}}
\newcommand{\op}{\Omega_{p}}

\newcommand{\Hbp}{\overline{H}_{p}}
\newcommand{\Hp}{H_{p}}

\newcommand{\supp}{\mbox{supp}}
\newcommand{\esup}{\mbox{ess sup}}

\newcommand{\jac}{\mbox{jac}\,}

\newcommand{\Bb}{ \overline{B}}

\newcommand{\R} {\mathbb{R}}

\newcommand{\N} {\mathbb{N}}
\newcommand{\C} {\mathcal{C}}
\newcommand{\A} {\mathcal{A}}

\newcommand{\hn}{\mathcal{H}}

\newcommand{\hti}{\tilde{h}}

\newcommand{\db}{\overline{d}}

\newcommand{\D}{\mathcal{D}}

\newcommand{\F}{\mathcal{F}}
\newcommand{\fpb}{\overline{\mathcal{F}}}

\newcommand{\E}{\mathcal{E}}

\newcommand{\Pp}{\mathbb{P}}

\newcommand{\ep}{\varepsilon}

\newcommand{\ux}{U_{x_0}}
\newcommand{\zx}{Z_{x_0}}
\newcommand{\nx}{N_{x_0}}

\newcommand{\K}{\mathcal{K}}

\newcommand{\ccn}{\mathcal{C}_n(R)}

\newcommand{\xt}{\tilde{x}}

\newcommand{\hh}{\mathcal{V}}

\newcommand{\ini}{\in [1,\infty]}

\newcommand{\ept}{\tilde{\ep}}

\newcommand{\bp}{\Big{(}}

\newcommand{\bq}{\Big{)}}
\newcommand{\jpp}{\frac{1}{p}}
\newcommand{\ori}{0_{\R^i}}

\newcommand{\orm}{0_{\R^m}}

\title[Poincar\'e duality for $L^p$ cohomology...] {Poincar\'e duality for   $L^p$ cohomology on subanalytic singular spaces}



\author{Guillaume Valette}

\address
{Institute of Mathematics, Jagiellonian University, ul. S. \L ojasiewicza 6, 30-348  Krak\'ow,
Poland} \email{gvalette@impan.pl}


\thanks{}


\begin{document}

\begin{abstract} 
We investigate the problem of Poincar\'e duality for $L^p$ differential forms on bounded subanalytic submanifolds
of $\R^n$ (not necessarily compact).  We show that, when $p$ is sufficiently close to $1$ then  the $L^p$
cohomology of such a submanifold is isomorphic to its singular
homology.   In the case where $p$ is large, we show that $L^p$ cohomology is dual to intersection homology. As a consequence, we can deduce that the $L^p$ cohomology is  Poincar\'e
dual to $L^q$ cohomology, if $p$ and $q$ are H\"older conjugate to each other and $p$ is sufficiently large.
\end{abstract}

\maketitle

\begin{section}{introduction}
The history of $L^p$ forms on singular varieties began when J.
Cheeger computed the $L^2$ cohomology
groups for varieties with metrically  conical singularities and started constructing a Hodge theory for  singular compact
varieties \cite{c1,c2, c3,c4}.  This enabled him to derive Poincar\'e duality results for singular varieties. These groups turned out to be related to  intersection cohomology \cite{cgm}, which clarified the interplay between Poincar\'e duality for $L^2$ cohomology and the geometry of the underlying variety. A significant achievement  was then made by W. C. Hsiang and V. Pati who proved  that the $L^2$ cohomology of  complex normal algebraic surfaces is isomorphic to intersection cohomology \cite{hp}.


%
%
%

 %


Since Cheeger's work on $L^2$ forms, many other authors have investigated
$L^p$ forms on singular varieties  focusing on various classes of Riemmanian
manifolds, with different restrictions on the metric near the
singularities, like in the case of the so-called $f$-horns  \cite{ bgm,d, weber,s1,s2,y}. In the present paper, assuming only
 that the given set is subanalytic (possibly singular) we investigate the problem of Poincar\'e duality for $L^p$ forms for $p$ sufficiently large or close to $1$. Our approach relies on the precise description of the Lipschitz geometry initiated by the author in \cite{vlt, vpams, vlinfty}.

In order to describe the achievements of this article, let us recall the de Rham theorem recently proved  in \cite{vlinfty} which motivated  the present paper. 

\begin{thm}\label{thm_intro_linfty}\cite{vlinfty} Let $X$ be a compact subanalytic pseudomanifold. For any $j$, we have:
$$H_\infty ^j(X_{reg}) \simeq I^{t}H^j (X).$$
Furthermore, the isomorphism is induced by the natural map
provided by integration on allowable simplices.
\end{thm}

 Here, $H_\infty ^j$ denotes the $L^\infty$ cohomology while $I^\bt H^j(X)$ stands for the intersection cohomology of
 $X$ in the maximal perversity.
  The definitions of these cohomology theories are recalled in sections \ref{sect_lp} and \ref{sect_ih} below.
  We write $X_{reg}$  for the nonsingular part of $X$,
  i.e. the set of points at which $X$ is a smooth manifold.
  
    Intersection homology was introduced by M. Goresky and R. MacPherson in order to investigate the topology of singular sets. What makes it  very attractive is that they showed in their fundamental
paper \cite{ih1} that it satisfies Poincar\'e duality for a quite
large class of sets (recalled in Theorem \ref{thm_poincare_ih} below)
enclosing all the complex analytic sets (see also \cite{ih2}).
The above theorem thus raises the very natural question whether we can hope for Poincar\'e duality for $L^\infty$ cohomology of subanalytic pseudomanifolds, or more generally for $L^p$ cohomology $p\in [1,\infty]$. 

The natural candidate for being dual to $L^p$ cohomology is $L^q$ cohomology with 
 $\frac{1}{p}+\frac{1}{q}=1$. 
We  start by proving a de Rham theorem for the $L^p$ cohomology of a subanalytic submanifold $M\subset \R^n$ in the case where $p$ is close to $1$ (Theorem \ref{thm_de_Rham_lq}). We  also generalize Theorem \ref{thm_intro_linfty} by proving a De Rham theorem for $L^p$ cohomology for $p$ sufficiently large (Theorem \ref{thm_de_Rham_linfinity}). These results can be regarded as subanalytic versions of Cheeger's theorems.


This
 enables us to establish some
Poincar\'e duality results for $L^p$ cohomology  (Corollaries \ref{cor_1} and \ref{cor_2}).
Intersection homology turns out to be very useful to
assess the lack of duality between $L^p$ and $L^q$ cohomology.
In particular, we  see that the obstruction for this duality to hold is of purely
topological nature. Although the $L^p$  condition is
closely related to the metric structure of the singularities, the
 theorems below show that the knowledge of the topology of the
singularities is enough to ensure Poincar\'e duality. It is worthy of
notice that the only data of the topology of $X_{reg}$ is not enough. 

\subsection*{Organization of the article.}In section \ref{sect_framework}, we set-up our framework, state our de Rham theorems for $L^p$ cohomology, and derive two corollaries about 
 Poincar\'e duality.  
 The proof of these de Rham theorems is postponed to section \ref{sect_proof_de_rham_lp}.

The strategy used to establish them in section \ref{sect_proof_de_rham_lp} is classical: we first establish some Poincar\'e Lemmas for $L^p$ cohomology (Lemmas \ref{Poincare_lem_l1_simple} and  \ref{lem x normal Poincare lemma}) and then conclude by a sheaf theoretic argument.
Our
Poincar\'e Lemmas for $L^p$ cohomology require to define some homotopy operators on $L^p$ forms. The construction of these operators  (see (\ref{eq_def_alpha}) and (\ref{eq_K_0_2eme_forme})) as well as the study of their properties is carried out in section \ref{sect_hom_hop}.

Because of the of metric nature of the $L^p$ condition, this requires a delicate study of the Lipschitz properties of subanalytic singularities, which is the subject matter of section \ref{sect_lip}.
 Using the techniques developed in  \cite{vlt,vpams,vlinfty}, we show that
the conical structure of subanalytic set-germs may be required to have nice Lipschitz properties (Theorem \ref{thm_local_conic_structure}).   This theorem, which is of its own
interest, improves significantly the results of  \cite{vlinfty} where it
was  shown that every subanalytic germ may be retracted in a
Lipschitz way.  Since the homeomorphism of the conical structure provided by Theorem \ref{thm_local_conic_structure} is not smooth but just subanalytic and Lipschitz (unlike in \cite{c3, y}),      we have problems to pull-back smooth differential forms to smooth ones and
we shall also require stratification theory (in sections \ref{sect_hor_C1}, \ref{sect_strat}, and \ref{sect_lip_lp}) to overcome these difficulties (the subanalytic character of the homeomorphism of Theorem \ref{thm_local_conic_structure} is therefore essential). 
%
   
   As  we will be working near a point $x_0$ that can be a singular point of the closure of our given manifold, the  operators of section \ref{sect_hom_hop} may  provide nonsmooth forms. Such technical problems actually already arose in Cheeger's original paper \cite{c1} as well as in other settings \cite{y}, and we shall rely upon similar techniques to overcome these difficulties,
  using the notion of weak differentiation and starting by constructing some regularization operators in section \ref{sect_weakly_smooth} (see Theorem \ref{thm_regularization_operators}). The construction of these regularization operators that we will follow is indeed due to G. de Rham \cite{derham} (see also \cite{gold}).

   \end{section}

\begin{section}{Framework and main results}\label{sect_framework}

\begin{subsection}{Some notations}Throughout this article, $m$, $n$ and $k$ will stand for  integers. By ``smooth'', we will mean $C^\infty$.

We denote by $|.|$ the Euclidean norm of $\R^n$.  Given $x\in  \R^n$ and $\ep>0$, we respectively denote by $S(x,\ep)$ and  $B(x,\ep)$ the sphere and the open ball of radius $\ep$ that are centered  at $x$ (for the Euclidean distance). We also write $\overline{B}(x,\ep)$ for  the corresponding closed ball. 
Given a subset $A$ of $\R^n$, we denote the closure of $A$ by $cl(A)$ and set $\delta A=cl(A)\setminus A$.

Given two functions $\xi$ and $\zeta$  defined on a subset $A$ of $\R^n$ and a subset $B$   of $A$, we write ``$\xi \lesssim \zeta$ on $B$'' if there is a constant $C$ such that $\xi(x) \le C\zeta(x)$, for all $x\in B$. We write ``$\xi \sim \zeta$ on $B$'' if we have both   $\xi \lesssim \zeta$ and $\zeta \lesssim \xi$ on $B$. 

The graph of a mapping $f:A\to B$ will  be denoted $\Gamma_f$. A mapping $\xi:A\to \R^k$, $A\subset \R^n$, is said to be {\it Lipschitz} if it is Lipschitz with respect to the metric $|.|$, i.e., if there is a constant $C$ such that  $|\xi(x)-\xi(x')|\le C |x-x'|$, for all $x$ and $x'$ in $A$. We will say {\it $C$-Lipschitz} if we wish to specify the constant.
\end{subsection}

\begin{subsection}{The subanalytic category}We now  recall some basic facts about subanalytic sets and functions.

 \begin{dfn}\label{dfn_subanalytic}Let $N$ be an analytic manifold.
A subset $E\subset N$ is called (locally) {\it  semi-analytic} if it is locally
defined by finitely many real analytic equalities and inequalities. More precisely, for each $a \in   N$, there is
a neighborhood $U$ of $a$, and real analytic  functions $f_i, g_{ij}$ on $U$, where $i = 1, \dots, r, j = 1, \dots , s$, such that
\begin{equation}\label{eq_definition_semi}
E \cap   U = \bigcup _{i=1}^r\bigcap _{j=1} ^s \{x \in U : g_{ij}(x) > 0 \mbox{ and } f_i(x) = 0\}.
\end{equation}

We denote by $\Pp_1$ the $1$-dimensional real projective space. Let $\hh:\R\to \Pp_1$ be the mapping defined by $\hh(x)=[1:x]\in\Pp_1$ for every $x\in \R$.   Every  subset of $\R^n$ may be regarded as a subset of $\Pp_1^n$ via the homeomorphism (onto its image)
 $$\hh^n : \R^n  \to \Pp_1 ^n,  $$
$$(y_1, \dots, y_n) \mapsto  (\hh(y_1),\dots,\hh(y_n)).$$
  A subset $Z$ of $\R^n$ is  {\it globally semi-analytic} if $\hh^n(Z)$ is  a  semi-analytic subset of $\Pp_1^n$.  Of course, globally semi-analytic sets are semi-analytic. Clearly,  a bounded subset of $\R^n$ is semi-analytic if and only if it is globally semi-analytic.
\end{dfn}


Working with {\it globally} semi-analytic sets will make it possible to avoid some pathological situations at infinity. In particular, it will enable us to work without any properness assumption.  
The function $\sin x$ is a typical example of a function which is  semi-analytic but not globally semi-analytic.

\medskip

\begin{dfn}
 A subset $E\subset \R^n$  is  {\it subanalytic} (resp.  {\it globally subanalytic}) if it
 can be represented as the projection of a semi-analytic (resp. globally semi-analytic) set; more precisely, if
there exists a semi-analytic (resp. globally semi-analytic)
set $Z \subset   \R^{n+p}$, $p\in \N$, such that $E  = \pi(Z)$, where $\pi :    \R^{n+p} \to    \R^n $ is the projection omitting the $p$ last coordinates. In particular,
globally semi-analytic sets are globally subanalytic.

%

   We say that {\it a mapping $f:A \to B$ is  subanalytic} (resp. {\it globally subanalytic}), $A \subset \R^n$, $B\subset \R^m$, if its graph is a subanalytic  (resp. globally subanalytic) subset of $\R^{n+m}$. In the case $B=\R$, we say that  $f$ is a {\it subanalytic} (resp. {\it globally subanalytic}) {\it function}.
\end{dfn}
The advantage of the  globally subanalytic category is that, unlike the globally semi-analytic category,  it is stable under linear projection. Globally  subanalytic sets constitute a nice category to study the geometry of semi-analytic sets: it is also stable under union, intersection, complement, and Cartesian product. Moreover, these sets enjoy many finiteness properties. For instance, they always have finitely many connected components, each of them being  globally subanalytic.

If $X$ is a subanalytic set then $X_{reg}$, which is the set of points at which $X$ is a $C^\infty$ manifold (of dimension $\dim X$ or smaller), is an open dense subanalytic subset of $X$. Another feature of the subanalytic category which will be important for our purpose is the famous \L ojasiewicz's inequality. We shall use it in the following form.
\begin{pro}
\label{pro_lojasiewicz_inequality}(\L ojasiewicz inequality)
Let $f$ and $g$ be two  globally subanalytic functions on a globally subanalytic set $A$. Assume that $f$ is bounded and that  
$\lim_{t\to 0}f(\gamma(t))=0,$
 for every globally subanalytic arc $\gamma:(0,\ep)\to A$ such that $\lim_{t\to 0} g(\gamma(t))=0$.
 Then there exist $N \in \N$ and $C \in \R$ such that for any $x \in A$:
\begin{equation}\label{eq_loj}
|f(x)|^N \leq C|g(x)|.\end{equation}
\end{pro}

This inequality originates in \cite{loj}. Several improvements were then obtained. This form is due to \cite{ania} (Proposition 1.1). We refer to \cite{dd,bm} for more about subanalytic sets.
\end{subsection}

\begin{subsection}{$L^p$ cohomology.}\label{sect_lp}
Let $M$ be a $C^\infty$ submanifold of $\R^n$. We   equip $M$  with the Riemannian metric inherited from the ambient
space, this set being endowed with the Euclidean inner product.  This metric gives rise to a measure $vol_M$.

Given a measurable function $f:M\to \R$ (the word {\it measurable} will always refer to this measure), we will denote by $\int_{x\in M} f(x)$ (or sometimes $\int_M f$) the integral of $f$ with respect to $vol_M$ (in other words, we will not match the measure in the notation).

For $p \in [1,\infty)$, we then say that {\it the function $f$ is $L^p$} if it is $L^p$ with respect to the measure $vol_M$, i.e., if $\int_{x\in M} |f(x)|^p <\infty$.   
We will write $|f|_p$ for the {\it $L^p$ norm of $f$} (possibly infinite), i.e., $|f|_p:=\int_{x\in M}|f(x)|^p.$ 

We say that $f$ is $L^\infty$ if there is a constant $C$ such that $|f(x)| \le C$ for almost  every $x\in M$.  The  $L^\infty$ norm of $f$ will be denoted $|f|_\infty$ and will be, as usual, the {\it essential supremum of $f$ on $M$}, i.e., 
$$|f|_\infty:=\esup_{x\in M} |f(x)|:=\inf \{a\in \R: vol_M(|f|^{-1}([a ,\infty))=0\}, $$   with the convention that this infimum is infinite if the considered set is empty.

Differentiable forms  will always be assumed to be at least measurable (i.e., giving rise to a measurable function when composed with a smooth section of multivectors).
Given a differential $j$-form $\omega$ on $M$, we will denote by $|\omega(x)|$ the norm of the linear mapping $\omega(x): \otimes^j T_x M\to \R$ with respect to the metric of $M$. As usual, we will denote by $d$ the exterior differentiation of forms  (for all manifolds and all $j$).

\begin{dfn} 
  Given $p \in [1,+\infty]$, we say that {\it  a differential $j$-form $\omega$
on $M$  is $L^p$}  if  the function $f(x):=|\omega(x)|$ is $L^p$. In the case where $p$ is finite, this means that  
 $$|\omega|_p:=\int_{x\in M} |\omega(x)|^p<\infty.$$
In the case $p=\infty$ this means that 
 there exists a constant $C$ such that: $$|\omega|_\infty:=\esup_{x\in M} |\omega(x)|<\infty.$$

 We denote by $\Omega^j_{p} (M)$ the real  vector space
constituted by the smooth $L^p$ differential $j$-forms $\omega$ on $M$ for which
 $d\omega$ is also $L^p$. The  cohomology groups of the  cochain complex $(\Omega^j_{p} (M) , d)_{j\in \N}$ 
  are called the {\it
$L^p$ cohomology groups of $M$} and will be denoted by
$H^j _{p}(M)$.
\end{dfn}

\end{subsection}

\begin{subsection}{Intersection homology.}\label{sect_ih}
 We recall the definition of intersection homology as it was introduced by Goresky and Macpherson \cite{ih1,ih2}.
\begin{dfn}\label{def_stratification}
Let $X\subset \R^n$ be a subanalytic subset. A {\it 
stratification of $X$} is a finite partition of this set into
 subanalytic $C^\infty$ submanifolds of $\R^n$, called {\it strata}.

  We now are going to define inductively on the dimension of $X$ the locally topologically trivial stratifications of $X$. For $\dim X=0$, every stratification is locally topologically trivial. 

 We denote by $cL$  the open cone over the space $L\subset \R^n$  of vertex at the origin,  $c\emptyset$ being reduced to the origin. Observe that if $L$ is  stratified by $\Sigma$ then $cL$ is stratified by   $cS\setminus \{0\}$, $S \in \Sigma$, and the origin.

 A stratification $\Sigma$ of $X$ is said to be {\it locally topologically trivial} if
 for every $x \in S \in \Sigma$, there is a
subanalytic homeomorphism $$h:U_x \to B(\ori,1) \times  cL,$$  (where $i=\dim S$) with $U_x$ neighborhood of $x$ in $X$ and $L\subset X\setminus \{x\}$  compact subanalytic subset  having a  locally topologically trivial  stratification  such  that
 $h$ maps the strata of $U_x$ (induced stratification) onto the strata of  $ B(\ori,1) \times cL$  (product stratification).
\end{dfn}

\begin{dfn}\label{del pseudomanifolds}
A subanalytic subset $X\subset \R^n$ is an  {\it $m$-dimensional subanalytic pseudomanifold} if
 $X_{reg}$ is an $m$-dimensional manifold  and  $\dim X\setminus X_{reg}<m-1$.

A {\it stratified pseudomanifold} (of dimension $m$) is the data of an $m$-dimensional subanalytic pseudomanifold $X$ together with  a locally topologically trivial stratification $\Sigma$ of $X$ having no stratum of dimension $(m-1)$. 
\end{dfn}

\begin{dfn}\label{dfn_ih}
Let $(X,\Sigma)$ be an  $m$-dimensional stratified
pseudomanifold and let $X_i$ denote the union of all the strata of $\Sigma$ of dimension less than or equal to $i$. 
A {\it perversity} is a sequence of integers $p = (p_2, p_3,\dots, p_m)$
such that $p_2 = 0$ and $p_{k+1} = p_k$ or $p_k + 1$.  Given a perversity $p$, a subset
$Y \subset X$ is called  {\it $(p,i)$-allowable} if for all $k$ $$\dim Y \cap
X_{m-k} \leq p_k+i-k.$$ 

  Define $I^{p}C_i (X)$ as the subgroup of
$C_i(X)$ consisting of those chains $\sigma$ such that $|\sigma|$
is $(p, i)$-allowable and $|\partial \sigma|$ is $(p, i -
1)$-allowable.
 The {\it $i^{th}$ intersection homology group of perversity $p$}, denoted
$I^{p}H_i (X)$, is the $i^{th}$ homology group of the chain
complex $I^{p}C_\bullet(X).$  The {\it $i^{th}$ intersection cohomology group of perversity $p$}, denoted
$I^{p}H^i (X)$, is defined as ${\bf Hom} (I^{p}H_i (X),\R)$.
\end{dfn}

In \cite{ih1,ih2} Goresky and MacPherson have proved that  these homology groups are
finitely generated and independent of the (locally topologically trivial) stratification. Since topologically trivial
stratifications exist for all subanalytic pseudomanifolds \cite{tesc} (Whitney $(b)$-regular stratifications do have this property), we
 will not always specify the chosen
stratification.

Moreover, Goresky and MacPherson also proved that their theory satisfy a
generalized version of Poincar\'e duality. We set $\bt:=(0,1,\dots,m-2)$.

\begin{thm}(Generalized Poincar\'e duality
\cite{ih1,ih2})\label{thm_poincare_ih}
Let $X$ be a compact oriented stratified pseudomanifold and let $p$ and $q$ be
perversities with $p+q=\bt$. For all $j$, we have:
$$I^p H^j(X)\simeq I^q H^{m-j}(X).$$
\end{thm}

\medskip

In particular, in  the case of the perversities $p=0=(0,\dots,0)$ and $q=\bt$, we get \begin{equation}\label{eq_gm_ot}
I^0 H^j(X)\simeq I^\bt H^{m-j}(X).\end{equation}
\end{subsection}

\begin{subsection}{The de Rham theorems.}
 In this section we  state our de Rham theorems. The proofs require technical preliminaries and will appear in section \ref{sect_proof_de_rham_lp}. 

\begin{thm}\label{thm_de_Rham_lq}Given a   bounded subanalytic submanifold $M$ of $ \R^n$, we have
for each  $p\in [1,\infty)$ sufficiently close to $1$ and each integer $j$: 
$$H_{p}^j(M)\simeq H^j(M).$$
\end{thm}

Such a theorem is of course no longer true without the subanalycity assumption.
 We will also show that the isomorphism is given by integration of forms on simplices. By ``$p$ sufficiently close to $1$'', we mean that there is $p_{0}\in (1,\infty]$ such that this statement holds for all $1\le p < p_{0}$. The bound for $p$ will be provided by the famous \L ojasiewicz's inequality (see Proposition \ref{pro_lojasiewicz_inequality}).

We then will improve Theorem \ref{thm_intro_linfty} by showing:

\begin{thm}\label{thm_de_Rham_linfinity}Given a   bounded  subanalytic $m$-dimensional submanifold $M$ of $ \R^n$ such that $\dim \delta M\le m-2$, we have
for all $p\in [1,\infty]$ large enough and each integer $j$: 
$$H_{p} ^j(M) \simeq I^{\bt} H^j  (X),$$
where $X$ denotes the closure of $M$.
\end{thm}

 Again, by ``$p$ large enough'', we   mean that there is $p_{0}\in [1,\infty)$ such that this statement holds for all $p > p_{0}$.
\end{subsection}

\begin{subsection}{$L^q$ as a Poincar\'e dual for $L^p$} Given  $p\in [1,\infty]$, we call the number  $q\in [1,\infty]$ that satisfies $\frac{1}{p}+\frac{1}{q}=1$,  the {\it H\"older conjugate of $p$}.
  Thanks to  Goresky
and MacPherson's generalized Poincar\'e duality, we can derive explicit
topological criteria on the singularity to determine whether
$L^p$ cohomology is Poincar\'e dual to $L^q$ cohomology.

\begin{cor}\label{cor_1}
Let $X$ be a compact oriented $m$-dimensional subanalytic pseudomanifold.  Take $p\in [1,\infty]$ and denote by  $q$ its H\"older conjugate. If $H^j(X_{reg})\simeq
I^0H^j(X)$ then, for $p$ sufficiently close to $1$, $L^p$ cohomology is Poincar\'e dual to $L^q$
cohomology in dimension $j$, in the sense that $$H^j_p(X_{reg})\simeq
H_{q}^{m-j}(X_{reg}).$$
\end{cor}
\begin{proof}
This  is a consequence of Theorems \ref{thm_de_Rham_lq},
\ref{thm_de_Rham_linfinity}, and Goresky and MacPherson's generalized
Poincar\'e duality (\ref{eq_gm_ot}).
\end{proof}

\medskip

\begin{cor}\label{cor_2}Let $M \subset \R^n$ be a bounded subanalytic oriented  $m$-dimensional $C^\infty$ submanifold.  Take $p\in [1,\infty]$ and denote by  $q$ its H\"older conjugate.
 If  $\dim \delta M=k$ then,  for $p$ sufficiently close to $1$, $L^p$ cohomology is Poincar\'e dual to  $L^q$
cohomology in dimension $j$, for each  $j<m-k-1$, i.e., for each such $j$:
$$H_{p} ^j(M) \simeq  H_q^{m-j} (M).$$
\end{cor}
\begin{proof} We may assume $k<m-1$ since otherwise the statement is vacuous. Observe that $X:=cl(M)$ is then a compact subanalytic pseudomanifold.
Fix a stratification that makes of $X$ a stratified pseudomanifold. We can choose this stratification compatible with $\delta M$ and such that there is no stratum $S$ satisfying  $ k < \dim S <m$ (all the strata of positive codimension may be assumed to be included in $\delta M$). By definition of
$0$-allowable chains (see section \ref{sect_ih}), the
support of a singular chain $\sigma \in I^0 C_{j} (X)$  may not
intersect the strata of the singular locus of dimension less than
$(m-j)$.  If $j<m-k$ (or equivalently $k<m-j$) $|\sigma|$ thus  must lie entirely in
$M$, which entails $$I^0C_{j} (X)=C_{j}(M).$$

If $j<m-k-1$ then the same applies to $(j+1)$ and consequently $$I^0H_{j} (X)=H_{j}(M).$$
 The result is therefore again a consequence of  Theorems \ref{thm_de_Rham_lq},
\ref{thm_de_Rham_linfinity}, and Goresky and MacPherson's generalized
Poincar\'e duality (\ref{eq_gm_ot}).
\end{proof}
\begin{rem}
It could be  seen that  this duality is provided by the natural pairing given by integration. Considering $L^p$ cohomology with bounded support, it is possible to generalize this duality to the case of unbounded manifolds.
\end{rem}
\end{subsection}

\end{section}

\begin{section}{On $L^p$ weakly differentiable forms.}\label{sect_weakly_smooth}
As emphasized in the introduction, we shall need 
to work with nonsmooth forms, which we will differentiate as distributions. We will have to check that the  weakly smooth forms give rise to the same cohomology groups as the smooth ones (Corollary \ref{cor_ws_isom_smooth}), which requires to construct regularization operators (Theorem \ref{thm_regularization_operators}). This kind of techniques is actually very classical and we will follow de Rham's construction  \cite{derham}.  The $L^p$ estimates are however due to \cite{gold}. We first provide definitions.

Throughout this section,  $M$ stands for a $C^\infty$ submanifold of $\R^n$ and $m$ for its
dimension. This manifold is not assumed to be orientable. For each $k$, $\R^k$ will be assumed to be endowed with the orientation given by the canonical basis.

\begin{subsection}{Weakly differentiable forms.}\label{subsect_weakly_smooth}
 We say that a differential form $\omega$ on $M$ is {\it $L^1_{loc}$} if the restriction of $\omega$ to every compact subset of $M$ is $L^1$.
We denote by $\Lambda_{0} ^{j}(M)$ the set of
$C^\infty$ $j$-forms on $M$ with compact support.



\begin{dfn} Let $U$ be an open subset of $\R^m$.
 An $L^1_{loc}$  differential
$j$-form $\alpha$ on $U$ is called {\it weakly differentiable}   if
there exists an $L^1_{loc}$ $(j+1)$-form $\omega$ on $U$ such that for any form
$\varphi\in \Lambda_{0} ^{m-j-1} (U)$:$$\int_{U} \alpha \wedge
d\varphi =(-1)^{j+1}\int_{U} \omega \wedge \varphi.$$ The form
$\omega$ is then called {\it the weak exterior differential of
$\alpha$} and we write $\omega=\overline{d} \alpha$. A form $\omega$ on $M$ is {\it weakly differentiable} if it gives rise to such a form via the coordinate systems of $M$ (and the weak  exterior differential is then obtained by  pulling-back the corresponding weak  exterior differential).

We will denote by $ \overline{\Omega}^j(M)$ the space of weakly differentiable $j$-forms on $M$. We also set $\overline{\Omega}(M)=\underset{j\le m}{\oplus} \overline{\Omega}^j(M)$.

 For every $p\in [1,\infty)$, we denote by $\overline{\Omega}_{p} ^{j}(M)$  the set of
weakly differentiable $j$-forms which are  $L^p$ and
which have an  $L^p$ weak exterior differential. Together with
$\overline{d}$, these $\R$-vector spaces constitute  a cochain complex. We denote by
$\overline{H}_{p}^j (M)$  the resulting cohomology groups.\end{dfn}

\end{subsection}

\begin{subsection}{Regularizing operators} 
 We write $\supp\, \varphi$  for the support of a form $\varphi$ on $M$ and $\Lambda_{or} ^{j}(M)$ for the space of forms $\varphi\in \Lambda_{0} ^{j}(M)$ for which $\supp\, \varphi$ fits in an oriented open subset of $M$.
For simplicity,  we set for $\omega \in \overline{\Omega}^j(M)$ and $\varphi\in \Lambda_{or} ^{m-j} (M)$ $$<\omega,\varphi>:=\int_{M} \omega \wedge
\varphi. $$
 Moreover, given  a vector field $\xi$ on $M$, we denote by $\omega_\xi$ the differential $(j-1)$-form defined by $\omega_\xi(x)(\zeta):=\omega(x)(\xi(x) \otimes \zeta)$ for $x\in M$ and $\zeta \in \otimes^{j-1} T_xM$.

\begin{thm}\label{thm_regularization_operators}
There exist sequences of linear operators $R_i$ and $A_i$, $i \in \N$, on $\overline{\Omega}(M)$ satisfying the following properties:
 \begin{enumerate}[(1)]
  \item\label{item_hom} If $\omega \in \ws^j(M)$ then $R_i\omega$ and $A_i\omega$ respectively belong to $\Omega^{j}(M)$ and $\overline{\Omega}^{j-1}(M)$ and satisfy:
  \begin{equation}\label{eq_ra}R_i \omega-\omega=\db A_i \omega+A_i\db \omega.\end{equation}
%
  \item\label{item_bd_lp} If $\omega$ is $L^p$, $A_i\omega$ and $R_i\omega$ are $L^p$ as well, $1\le p\le \infty$. 
  \item \label{item_cv_lp}For each $L^p$ form $\omega$,   $1\le p<\infty$, $R_i \omega\to \omega$ and $A_i \omega\to 0$ for the $L^p$ norm, as $i \to \infty$.
  \item\label{item_support} If $W$ is a neighborhood of the support of $\omega \in \ws(M)$ then the supports of $R_i \omega$ and $A_i\omega$ are included in $W$ for all $i$ sufficiently large. 
 \end{enumerate}
\end{thm}


In order to outline the regularization process, let us recall the way it is usually achieved in $\R^n$. On Euclidean spaces, regularization is achieved by means of a convolution product. Namely,  take a $C^\infty$ function $f:\R^m \to \R$ satisfying $\supp f\subset B(\orm,1)$ as well as $\int_{\R^m} f =1$, and set $f_\ep(x):=\frac{1}{\ep^m}f(\frac{x}{\ep})$.  If $\omega\in \ws(\R^m)$, we then can define an operator by:
$$R_\ep \omega(x) = \omega * f_\ep(x) =\int_{y\in\R^m} f_\ep (y)\omega(x-y) =\int_{y\in \R^m} f_\ep (y)\cdot\lambda^*_y\omega(x)  ,  $$
where $\lambda:\R^m\times \R^m\to \R^m$ is defined by  $\lambda_y(x):=x-y$.

It is then well-known that $R_\ep\omega$ is $C^\infty$, and, in the case where $\omega$ is $L^p$, that $R_\ep \omega$ tends to $\omega$ for the $L^p$ norm as $\ep \to 0$. Hence, if we set $R_i:= R_{\ep_i}$, where $\ep_i$ is a sequence which is tending to zero, we have a sequence of regularizing operators.  

 Generalizing this kind of techniques to manifolds requires  an adequate function $\lambda_y(x)$. This will force us to work locally (using a coordinate system). The purpose of Lemma \ref{lem_application_mu} below is to provide  a function which has the required properties in a neighborhood $U_{x_0}$ of a given point $x_0$.  The desired regularizing operator $R_i$ and $A_i$ will then be obtained by gluing the respective ``local regularizing operators'' associated to a locally finite covering of $M$ by coordinate systems. 


\begin{lem}\label{lem_application_mu}
 Given $x_0\in M$, there exists a $C^\infty$ diffeomorphism $h:U_{x_0} \to \R^m$, where $U_{x_0}$ is an open neighborhood of $x_0$ in $M$,  such that the family of  mappings $\sigma_y: M \to M$, $y\in \R^m$, defined by   
 $$\sigma_y(x)=\left\{
    \begin{array}{l}
     h^{-1}(h(x)+y), \quad\mbox{if } x \in U_{x_0},\\
      x, \hskip 2.5cm\mbox{if } x \in M\setminus U_{x_0},
    \end{array}
    \right.$$  is a family of $C^\infty$ diffeomorphisms, smooth with respect to $y$.
  
\end{lem}
\begin{proof}
Choose a  $C^\infty$  function $\phi:[0,1)\to \R$ which is equal to the identity on the interval $[0,\frac{1}{2}]$ and which coincides with the function $e^{\frac{1}{(1-t)}}$ near $1$.  We may take such a function that has positive first derivative everywhere.  Define then a $C^\infty$ diffeomorphism $\Phi:B(\orm,1)\to \R^m$ by setting $\Phi(x)=\phi(|x|)\frac{x}{|x|}$.

Take then a coordinate system $\Psi:U_{x_0}\to B(\orm,1)$ of $M$ near $x_0$.  We can assume that $U_{x_0}$ is relatively compact in $M$ and that $\Psi$ extends to a diffeomorphism on a neighborhood of $cl(U_{x_0})$ in $M$. Set  $h=\Phi \circ \Psi$ and define $\sigma_y$ as in the statement of the lemma. 

Since $\phi(t)$ and all its subsequent partial derivatives tend to infinity when $t$ tends to $1$, it is easy to see that the mapping $d_z h^{-1}=d_{\Phi^{-1}(z)}\Psi^{-1}\cdot d_z\Phi^{-1}$ and all its derivatives tend to zero when $z$ goes to infinity ($d\Psi^{-1}$ is bounded since $\Psi$ extends smoothly to a neighborhood of $cl(U_{x_0})$ and $U_{x_0}$ is relatively compact).
As, by Taylor's formula, we have 
$$\sigma_y(x)=  h^{-1}(h(x)+y)=x+\int_0 ^1 d_{h(x)+ty} h^{-1} \cdot y\,dt,$$
 we see that this entails that $\sigma_y$ is smooth at the boundary points of $U_{x_0}$. Smoothness with respect to $y$ follows from the same argument. The mapping $\sigma_y$ is a diffeomorphism since its inverse is given by $\sigma_{-y}$. 
\end{proof}
\begin{rem}\label{rem_u_j}
In the proof of the above lemma, the neighborhood $U_{x_0}$ is provided by a coordinate system. It thus can be chosen as small as we wish. 
\end{rem}

We are now ready to establish Theorem \ref{thm_regularization_operators}.



 \begin{proof}[Proof of Theorem \ref{thm_regularization_operators}]
The proof is divided in three steps. In step $1$  we define an operator $R_\ep$ regularizing the forms in a neighborhood of a given point $x_0$, as described in the paragraph preceding Lemma \ref{lem_application_mu}. In Step $2$, we define analogously an  operator $A_\ep$. In Step $3$, we glue the respective operators, constructed in the preceding steps for each element of a locally finite covering of $M$, into global operators $R_i$ and $A_i$.  

 Fix $x_0\in M$, denote by $\sigma_y: M\to M$, $y\in \R^m$,  the family of mappings provided by Lemma \ref{lem_application_mu}, and let  $U_{x_0}$ be the corresponding open neighborhood. We may assume that this neighborhood is relatively compact in $M$ (see Remark \ref{rem_u_j}). There is a positive constant $\lambda$ such that for all $\ep\in [0,1]$,  $y \in B(\orm,\ep)$,  and  $x$ in $M$, we have:  \begin{equation}\label{eq_sigma_id}|\sigma_y(x)-x| \le \lambda \ep.\end{equation}

 Take a $C^\infty$ function   $f:\R^n \to [0,1]$  satisfying $\supp f \subset B(\orm,1)$ and   $\int_{M} f=1$ as well as $f(y)=f(-y)$ for all $y$, and set $f_\ep(y):=\frac{1}{\ep^m}f(\frac{y}{\ep})$.

\noindent {\bf Step 1.} We define an operator $R_\ep$, $\ep>0$ (it is enough for our purpose to define $R_\ep$ for $\ep$ small), smoothing the forms on $U_{x_0}$.

Given  $\omega \in \ws(M)$, 
 we set for  all $x\in M$: 
\begin{equation}\label{eq_Rep_int}
 R_\ep \omega(x) = \int_{y\in\R^m} f_\ep (y)\cdot\sigma^{*}_y\omega(x) . 
\end{equation}


We now check that $R_\ep \omega$ is $C^\infty$ at every $x\in U_{x_0}$ for all $\omega \in \ws(M)$. The proof is actually analogous to the classical argument that shows that convolutions with a bump function are smooth. Namely,  for $x\in U_{x_0}$, by definition of the pull-back, (\ref{eq_Rep_int}) amounts to
$$ R_\ep \omega(x) = \int_{y\in \R^m} f_\ep (y)\cdot\omega(\sigma_y(x)) d_x\sigma_y, $$
so that if we make the substitution $z=y+h(x)$ (see Lemma \ref{lem_application_mu} for $h$) inside this integral, we get  
$$ R_\ep \omega(x) = \int_{z\in \R^m} f_\ep (z-h(x))\cdot\omega(h^{-1}(z)) d_x\sigma_{z-h(x)}, $$
which is clearly smooth with respect to $x$ (since so are $h, f_\ep$, and $\sigma$). 

 We now claim that $R_\ep$ is self-adjoint, that is to say, for every $\varphi \in \Lambda_{or} ^{m-j}(M)$ and every form $\omega \in  \overline{\Omega}^j(M)$, we have
\begin{equation}\label{eq_r_selfadjoint}
 <R_\ep \omega ,\varphi>=<\omega,R_\ep\varphi>.
   \end{equation}
Indeed, for such forms  $\varphi$ and $\omega$ we can write (since $\sigma_{-y}=(\sigma_y)^{-1}$),$$ <R_\ep \omega ,\varphi>= \int_{y\in \R^m}\int_{M}f_\ep(y)\cdot \sigma_y^*\omega\wedge \varphi= \int_{y\in \R^m}\int_{M} f_\ep(y)\cdot \omega \wedge \sigma_{-y}^*\varphi. $$
Making the substitution $z=-y$ in this integral then yields (\ref{eq_r_selfadjoint})  (since $f_\ep(-y)=f_\ep(y)$).

We  also need to establish an estimate of the $L^p$ norm of $R_\ep$.  Namely, we are going to show that  there exists a constant $C_{\ep}$   (depending on $\ep$ and not on $p\in [1,\infty]$) such that for every $L^p$ form $\omega \in \ws(U)$ we have:
 \begin{equation}\label{eq_estimate_l1_norm_Rep}
|R_\ep  \omega|_p \le C_{\ep}|\omega|_p .   
 \end{equation}
 Moreover,   we shall show that $C_{\ep}$   can be chosen in such a way that $C_{\ep}\to 1$  
  as $\ep \to 0$.

 To prove (\ref{eq_estimate_l1_norm_Rep}), fix $p\in [1,+\infty)$ (we postpone the case $p=\infty$). Observe that by definition of the pull-back of differential forms, we have  for $(x,y) \in M\times B(\orm,\ep)$:
 \begin{equation}\label{eq_rjjep_int}
|\sigma^{*}_y\omega(x)|=  |\omega(\sigma_y(x))d_x\sigma_y| \le c|\omega(\sigma_y(x))|, 
 \end{equation}
where $c=\sup \{ |d_x\sigma_y|: (x,y)\in M\times B(\orm,\ep) \}$. 
We deduce 
 \begin{equation}\label{eq_rjjep_int2}
|R_\ep \omega|_p \overset{(\ref{eq_rjjep_int})}{\le} c \big{|}\int_{y\in \R^m} f_\ep (y)\cdot|\omega\circ \sigma_y| \, \big{|}_p\le c \int_{y\in\R^m} f_\ep (y)\cdot|\omega\circ \sigma_y|_p ,
 \end{equation}
thanks to Minkowski's inequality. 

As $\sigma_y$ is a diffeomorphism, we can estimate $ |\omega\circ \sigma_y|_p$ by  making the substitution $u:=\sigma _y(x)$  as follows: 
$$ |\omega\circ \sigma_y|_p = \big{(}\int_{x\in M} |\omega\circ \sigma_y(x)|^p  \big{)}^{\frac{1}{p}}\le   \big{(}\int_{u\in M} |\omega(u)|^p\cdot \jac \sigma_y^{-1}(u) \big{)}^{\frac{1}{p}}\le c' |\omega|_p,$$
where $c'=\sup \{\jac \sigma_y^{-1}(u):(y,u)\in B(\orm,\ep) \times M\}$.
By (\ref{eq_rjjep_int2}), this entails:
 $$|R_\ep \omega|_p \le cc'|\omega|_p \int_{y\in\R^m}  f_\ep(y) = cc' |\omega|_p.$$
This yields (\ref{eq_estimate_l1_norm_Rep}) for all $p\in [1,\infty)$.  As $\sigma_y$ is smooth with respect to $y$,  it tends uniformly to the identity as $y$ goes to $0$. This implies that $c$ and $c'$ both go to $1$ as $\ep$ goes to zero.

We now address the case $p=\infty$. In this case,  we have:
$$|R_\ep \omega(x)|  \overset{(\ref{eq_rjjep_int})}{\le} c \int_{y\in \R^m} f_\ep (y) |\omega(\sigma_y(x))|  \le c|\omega|_\infty\int_{y\in \R^m} f_\ep (y) =c|\omega|_\infty,$$
which yields (\ref{eq_estimate_l1_norm_Rep}).
 
 We shall also need an estimate of the support of $R_\ep \omega$ that we give now.
 Note for this purpose that $\sigma_y(x)=x$ for all $x\in M\setminus \ux$ (for all $y$), so that   $R_\ep \omega (x)=\omega(x)$, for all $x $ in this set.
 Thus, thanks to (\ref{eq_sigma_id}), it is easily derived from the definition of  $R_\ep$ that: \begin{equation}\label{eq_support}
\supp\, R_\ep \omega\subset\{x\in M:d(x,\supp \,\omega)\le \lambda\ep\}. \end{equation}
 \medskip
 
\noindent {\bf Step 2.} We now turn to define our operator $A_\ep$.  Let us first define a mapping $\tau:\R^m\times [0,1]\times  M \to M$ by setting $\tau_y(t,x):=\sigma_{ty}(x)$ for $(y,t,x)\in\R^m\times [0,1]\times  M$.  Set then for $\omega\in\ws (M)$  
\begin{equation}\label{eq_Aep_int}
 A_\ep \omega(x) = \int_{y\in \R^m}\int_0^1 f_\ep (y)\cdot \tau^{*}_{y}\omega_{\pa_t}(x,t) dt , 
\end{equation}
where $\pa_t$ denotes the vector field $(1,0)$ on $ [0,1]\times  M$.

We now are going to establish some properties for $A_\ep$, analogous to those we proved for $R_\ep$. First, we notice that the argument we used to show (\ref{eq_estimate_l1_norm_Rep}) applies (replacing (\ref{eq_Rep_int})  with  (\ref{eq_Aep_int}) and $\sigma_y(x)$ with $\tau_y(x,t)$) to establish that there is a constant  $\eta_{\ep}$ such that  for every  $L^p$ form $\omega \in \ws(M)$ we have:
 \begin{equation}\label{eq_estimate_l1_norm_Aep}
|A_\ep  \omega|_p \le \eta_{\ep} |\omega|_p  .  
 \end{equation}
   Moreover, since the derivative of $\tau_y$ with respect to $t$ tends to zero as $y$ goes to zero, the obtained constant $\eta_\ep$ can be required to go to zero as $\ep$ goes to zero.


The form $A_\ep\omega$ can fail to be differentiable. To prove that it is  weakly differentiable, we first check that it is $L^1_{loc}$ for every $\omega\in\ws (M)$. As we  have  $A_\ep \omega (x)=\omega(x)$, for all $x \in M\setminus U_{x_0}$, the result  is clear locally at the points of  $M\setminus cl(U_{x_0})$. Hence, we just have to focus on the points of $cl(U_{x_0})$.  Take a compactly supported $C^\infty$ function $\phi$  which is equal to $1$ on $U_{x_0}$. Remark that $A_\ep(\phi \omega)(x) =A_\ep\omega (x)$, for all  $x\in K$. Consequently,  by (\ref{eq_estimate_l1_norm_Aep}),
 $$|A_\ep \omega_{|K}|_1\le |A_\ep(\phi \omega)|_1 \le  \eta_\ep |\phi \omega|_1 <\infty, $$
since $\omega$ is $L^1_{loc}$. This shows that $A_\ep \omega$ is $L^1_{loc}$.

The argument we used to establish (\ref{eq_r_selfadjoint}) can now be used  (replacing $\sigma_y$ with $\tau_y$ in its proof) in order to prove that for every compactly supported $\varphi \in \Lambda_{or} ^{m-j}(M)$ we have: \begin{equation}\label{eq_a_selfadjoint} <A_\ep \omega ,\varphi>=<\omega, A_\ep\varphi>.\end{equation}
 Note that it is clear from the definition of $A_\ep$ that $A_\ep\varphi$ is smooth  (since so are $\varphi$, $f_\ep$, and $\tau$) and that $A_\ep d\varphi=dA_\ep \varphi$. Thanks to (\ref{eq_a_selfadjoint}), we deduce that $A_\ep \omega\in \ws(M) $ for all $\omega\in\ws(M)$ and $\db A_\ep \omega=A_\ep \db\omega$.
 
Moreover, by  the respective definitions of $R_\ep$ and $A_\ep$, we immediately see that for every $C^\infty$ compactly supported form $\varphi$ we have
\begin{equation}\label{eq_ra_smooth}
 R_\ep \varphi - \varphi=dA_\ep \varphi+A_\ep d\varphi.
\end{equation}
By (\ref{eq_r_selfadjoint}) and (\ref{eq_a_selfadjoint}), we conclude that  for every   $\omega\in \ws(M)$ we have:
\begin{equation}\label{eq_ra_ws}
  R_\ep \omega -\omega=\db A_\ep \omega+A_\ep \db\omega.
\end{equation}

  \medskip

\noindent {\bf Step $3$.} We are ready to define the operators claimed in  the theorem.  Every point $x_0\in M$ has a neighborhood $U_{x_0}$ in which we can define the operators of steps $1$ and $2$.  
 As the family $U_{x_0}, \; x_0\in M$, covers $M$, there is a locally finite subfamily $(U_{x_k})_{k=1}^\infty$. 
 For each positive integer $k$ and every $\ep>0$,  we thus have some operators $A_\ep^k$ and $R_\ep^k$   such that $R_\ep ^k\omega$ is smooth on $U_{x_k}$
 for every   $\omega\in \ws(M)$. 

Take then a sequence $\tilde{\ep}=(\ep_k)_{k=1}^\infty$ of positive real numbers and set for $k \ge 1$
$$R^k_{\tilde{\ep}}=   R^k_{\ep_k}\cdots R^{2}_{\ep_{2}} R^1_{\ep_1}\quad \mbox{and} \quad A^k_{\tilde{\ep}}=A^k_{\ep_k}R^{k-1}_{\tilde{\ep}}, $$
with by convention $R^{0} _{\tilde{\ep}}=Id$. Set then
$$R_{\tilde{\ep}}=\lim_{k\to +\infty}R^k_{\tilde{\ep}} \quad \mbox{and} \quad A_{\tilde{\ep}}=\sum_{k=1} ^{+\infty} A_{\tilde{\ep}} ^k\,.$$
Inequalities (\ref{eq_estimate_l1_norm_Rep}) and (\ref{eq_estimate_l1_norm_Aep}) yield that for every positive integer $i$ there is a sequence $\tilde{\ep}_i=(\ep_{i,k})_{k=1}^\infty$ such that for all $p$
$$\prod_{k=1}^{+\infty}|R_{\ep_{i,k}}^k|_p\le 1+\frac{1}{i} \quad \mbox{and}\quad \sum_{k=1}^{+\infty} |R^{k-1}_{\ep_{i,k-1}}|_p \cdots |R^1_{\ep_{i,1}}|_p  |A_{\ep_{i,k}} ^k|_p\le \frac{1}{i}.$$

Let us check that the  operators $R_i:=R_{\ept_i} $ and $A_i:= A_{\tilde{\ep}_i}$ possess the required properties $(1-4)$.  Fix a positive integer $i$.  Note first that since, for each weakly differentiable form $\alpha$, $R_{\ep_{i,k}} ^k\alpha$ is smooth on $U_{x_k}$ as well as in the vicinity of each point of the complement of $U_{x_k}$ around which $\alpha$ is smooth, the operator $R_{\tilde{\ep}_i}$ clearly transforms weakly differentiable forms into smooth forms.
 By (\ref{eq_ra_ws}), we have for every $k$ 
 $$R_{\ept_i} ^k -R_{\ept_i} ^{k-1}=\db A_{\ept_i}^k+ A^k_{\ept_i}\db. $$
 Hence, adding all these formulas for $k=1,2\dots$, we get (\ref{eq_ra}), yielding (\ref{item_hom}).
 
 
 The choice we made on $\ept_i$ implies that for all $p$ \begin{equation}\label{eq_A_iR_i_i}|R_i|_p \le 1+\frac{1}{i} \quad \mbox{and}\quad |A_i|_p\le \frac{1}{i},\end{equation}
     establishing (\ref{item_bd_lp}). 
 
  By (\ref{eq_A_iR_i_i}), it is clear that  $|A_i|_p\to 0$ as $i \to \infty$. To complete the proof of (\ref{item_cv_lp}), we thus just have to check that $|R_i\omega-\omega|_p\to 0$ for any $L^p$ form $\omega$, for $p \in [1,\infty)$.  As for such $p$, the set of 
  $C^\infty$ compactly supported forms is dense in $(\ws^j_p(M), |. |_p)_{j\in \N}$,  and since $R_i$ is bounded for the $L^p$ norm,  it is enough to check that $R_i \varphi \to \varphi$ for the $L^p$ norm for every $C^\infty$ compactly supported form $\varphi$. 
  
  Fix such a form $\varphi$ and observe that, by (\ref{eq_support}), if all the $\tilde{\ep}_{i}$, $i\in \N$, are decreasing sufficiently fast, the support of all the forms $R^k_{\tilde{\ep}_i}\varphi$, $i\in  \N$, $k\in  \N$, will be comprised in some compact set $K$. Since $(U_{x_k})_{k=1}^\infty$ is locally finite, no $U_{x_k}$ can meet $K$ when $k$ is sufficiently large.  It means that $R_{\ept_i}^{k_0}\varphi=R_{\ept_i} \varphi$, for some $k_0$ large. But since
     $R^{k_0}_{\ep_{i}}\varphi$ tends to $\varphi$  uniformly (on compact sets) as $i\to \infty$, we see that $R_{\ept_i}\varphi$ tends to $\varphi$ uniformly,  yielding that $R_i\varphi$ tends to $\varphi$ for the $L^p$ norm.
     
     Thanks to (\ref{eq_support}), it is easily seen that we can require the support of $R_i \omega $ to be arbitrarily close to the support of $\omega$ by taking $i$ sufficiently large. A similar argument yields the analogous fact  for $A_i$, establishing (\ref{item_support}).  
 \end{proof}
 \begin{rem}\label{rem_regularization} The following observations will be useful.  Let $\omega \in \ws(M)$.
 \begin{enumerate}[(i)]
 \item\label{item_rd}  Applying (\ref{eq_ra}) to $\db\omega$ we get 
$$R_i\db \omega-\db\omega=\db A_i \db \omega.$$ 
Moreover, applying $\db$ to (\ref{eq_ra}), we see that $dR_i \omega-\db\omega= \db A_i\db \omega$. Together with the preceding equality, this entails:
\begin{equation}\label{eq_rd}R_i\db\omega=dR_i\omega.\end{equation}
   \item\label{item_rem_bd_lp} If $\db\omega$ is $L^p$, for some $p\in [1,\infty]$, then, by (\ref{eq_rd}) and (\ref{item_bd_lp}) of Theorem \ref{thm_regularization_operators}, $d R_i\omega=R_i\db \omega$ is $L^p$. Moreover, in the case where $\omega$ is $L^p$ as well, by (\ref{eq_ra}), we can then conclude that $\db A_i\omega$ is $L^p$.
\item If $\omega $ is $C^k$, for some $k\in \N$, then so is $A_i\omega$. This follows  from (\ref{eq_Aep_int}).
 \end{enumerate}
\end{rem}

As a byproduct of the existence of regularizing operators, we get the following result that will be needed to establish our de Rham theorems for $L^p$ cohomology:
\begin{cor}\label{cor_ws_isom_smooth}For all $p \in [1,\infty]$, the inclusions $\op^j(M)\hookrightarrow \wsp^j(M)$, $j\in \N$, induce isomorphisms in cohomology.
%
\end{cor}
 \begin{proof} Fix $p\in [1,\infty]$ and, for $j\in \N$, denote by $\Lambda^j:\Hp^j(M)\to \Hbp^j(M)$ the mapping induced by the inclusion between the two cochain complexes. Let $R_i:\ws(M)\to \Omega(M)$ be the regularizing operator provided by Theorem \ref{thm_regularization_operators}. By Remark \ref{rem_regularization}, $R_i$ induces for every $j$ a mapping $R^j_i:\Hbp^j(M) \to \Hp^j(M) $,  which, due to (\ref{eq_ra}),  is nothing but the inverse of $\Lambda^j$.
    \end{proof}
 \end{subsection}



\end{section}

\begin{section}{Lipschitz properties of subanalytic sets and mappings}\label{sect_lip}
\begin{subsection}{Lipschitz conic structure  of  subanalytic set-germs.}The study of the metric geometry of singularities is more challenging than  the study of their topology.    For instance it is well-known that
subanalytic sets can be triangulated and hence are locally homeomorphic to
cones.  The situation is more complicated if one is interested in the
description of the aspect of singularities from the metric point of
view.   We however are going to prove that this conic structure may be required to have some nice metric properties (Theorem \ref{thm_local_conic_structure}) that will make it possible to establish our de Rham theorems later on.

  \begin{dfn}\label{dfn_cell_decomposition}
A {\it cell decomposition of $\R^n$} is a finite partition of $\R^n$ into globally subanalytic sets
$(C_i)_{i \in I}$, called {\it cells}, satisfying certain properties explained below.

$n = 1:$ A {\it cell decomposition of $\R$} is given by a finite subdivision $a_1 < \dots < a_l$ of
$\R$. The cells of $\R$ are the singletons $\{a_i\}$, $1\le  i \leq l$, and the intervals $(a_i,
a_{i+1})$,
$0 \leq  i \leq l$, where $a_0 = -\infty$ and $a_{l+1} = +\infty$.

$n > 1:$ A {\it cell decomposition of $\R^n$} is the data of a cell decomposition of $\R^{n-1}$ and,
for each cell D of $\R^{n-1}$, some $C^\infty$ globally subanalytic  functions   on $D$ (which, by induction, is a $C^\infty$ manifold):
$$\zeta_{D,1} < ... < \zeta_{D,l(D)} : D \to \R.$$

The {\it cells of $\R^n$} are the {\it graphs}
$$\{(x, \zeta_{D,i}(x)) : x \in D\}  ,$$
and the {\it bands}
$$(\zeta_{D,i}, \zeta_{D,i+1}) := \{(x, y) : x \in D
\mbox{ and }\zeta_{D,i}(x) <y<\zeta_{D,i+1}(x)\},$$
for $0 \leq  i \leq l(D)$, where $\zeta_{D,0}(x) = -\infty$ and $\zeta_{D,l(D)+1}(x) = +\infty$.

A cell decomposition is said to be {\it compatible with finitely many sets $A_1,\dots,A_k$} if the
$A_i$'s are unions of cells.
\end{dfn}

It is well-known that given some globally subanalytic sets  $A_1,\dots,A_k$, it is always possible to find a cell decomposition compatible with this family of sets. This fact is true on every o-minimal structure. A detailed proof in this framework can be found in \cite{costebleu}.

Let $\pi:\R^n\to \R^{n-1}$ be the projection omitting the last coordinate.  If $D$ is a cell of $\R^n$, we call $E:=\pi(D)$, the {\it basis} of $D$.  Observe that if $\D$ is a cell decomposition of $\R^n$, then $\pi(\D):=\{ \pi(D):D\in \D\}$ is a cell decomposition of $\R^{n-1}$.


\begin{dfn}
Let $A, B \subset \R^n$. A globally subanalytic map $h:A \to B$  is {\it $x_1$-preserving} if it
preserves the first coordinate in the canonical basis of $\R^n$, i.e., if for any $x=(x_1,\dots,x_n) \in A$, $\mu(h(x))=x_1$,   where $\mu:\R^n \to \R$ is the orthogonal projection onto the first coordinate.
\end{dfn}

If $R$ is a positive real number and $n$ a positive integer,  we set   $$\ccn :=\{x=(x_1,\dots,x_n) \in [0,\infty)\times \R^{n-1}: |x|\leq Rx_1\}. $$ 
We also set $\mathcal{C}_1(R):=[0,+\infty)$. 

We shall need the following lemma which was proved in \cite{vlinfty} (Lemma 2.2.3) to compute $L^\infty$ cohomology. In this lemma all the germs are germs at the origin.

\begin{lem}\label{prop proj reg}
Let  $A_1,\dots,A_\mu$ be germs of subanalytic subsets of $\C_{n}(R)$,  $R>0$, and 
$\eta_1,\dots,\eta_l$  be   germs of nonnegative globally subanalytic functions on $\C_{n}(R)$. There exist a germ of subanalytic
$x_1$-preserving bi-Lipschitz
  homeomorphism (onto its image) $\Phi:(\C_n(R),0) \to (\C_n(R),0)$ and a cell decomposition $\mathcal{D}$ of $\R^n$ such
  that:
  \begin{enumerate}[(i)]
  \item\label{item_compatible}  $\D$ is compatible with (some representatives of the germs) $\Phi(A_1),\dots, \Phi(A_\mu)$.
 \item\label{item_graphe} Every cell of $\D$ 
   which is a graph (i.e., not a band, see Definition \ref{dfn_cell_decomposition}) is the graph of a  function that has bounded derivative. 
\item\label{item_eq} On each
 $D\in \mathcal{D}$, every germ  $\eta_i\circ \Phi^{-1}(x)$ is $\sim$
to the germ of a function of the form:\begin{equation}\label{eq
prep}|x_n-\theta(\xt)|^r a(\xt)\end{equation} (for $x=(\xt,x_n) \in \R^{n-1}
\times \R$) where
$a,\theta:E \to \R$ are globally subanalytic  functions on the basis $E$ of $D$ with $\theta$ Lipschitz and $r \in \Q$. 
\end{enumerate}
\end{lem}

\begin{rem}\label{rem_bded_der_lips}
 In (\ref{item_graphe}), it is required that the function defining the cells have bounded derivative. Such functions are not necessarily Lipschitz (with respect to the Euclidean distance). They are nevertheless Lipschitz with respect to the so-called inner metric (given by the shortest  path joining two points). It  follows from the existence of $L$-reqular cell decompositions \cite{k} that  there is a partition of the basis of each cell in such a way that the two metric be equivalent on each element of this partition. This means that in (\ref{item_graphe}), we could require the functions defining the cells to be Lipschitz (with respect to the Euclidean distance).
\end{rem}

 Given $A\subset \R^n$ and $x_0\in \R^n$, we denote by $x_0*A$  the cone over the space $A$ of vertex $x_0$, i.e., the set of points of type $tx_0+(1-t)y$ with $y\in A$ and  $t\in [0,1]$  (by convention, $x_0*\emptyset$ will be reduced to the point $x_0$). 

\begin{thm}\label{thm_local_conic_structure}
  Let  $X\subset \R^n$ be subanalytic and $x_0\in X $. 
For $\ep>0$ small enough, there exists a Lipschitz subanalytic homeomorphism
$$H: x_0* (S(x_0,\ep)\cap X)\to  \Bb(x_0,\ep) \cap X,$$  
  satisfying $H_{| S(x_0,\ep)\cap X}=Id$, preserving the distance to $x_0$, and having the following metric properties:
\begin{enumerate}[(i)] 
 \item\label{item_H_bi}     The natural retraction by deformation onto $x_0$ $$r:[0,1]\times  \Bb(x_0,\ep)\cap X \to \Bb(x_0,\ep)\cap X,$$ defined by $$r(s,x):=H(sH^{-1}(x)+(1-s)x_0),$$ is Lipschitz.   
 Indeed, there is a constant $C$ such that  for every fixed $s\in [0,1]$, the mapping $r_s$ defined by $x\mapsto r_s(x):=r(s,x)$, is $Cs$-Lipschitz.
 \item \label{item_r_bi}  For each $\delta>0$,
 the restriction of $H^{-1}$ to $\{x\in X:\delta \le |x-x_0|\le \ep\}$ is Lipschitz and, for each $s\in (0,1]$, the map  $r_s^{-1}:\Bb(x_0,s\ep) \cap X\to \Bb(x_0,\ep) \cap X$ is Lipschitz. 
\end{enumerate}
\end{thm}

\begin{proof}
For $\ep>0$, let us set $$\C_n(R,\ep):=\{x=(x_1,\dots,x_n) \in \C_n(R): 0 \le x_1 \le \ep\} .$$The idea is to replace the distance to $x_0$ with the function given by the projection onto the first coordinate. We will
 prove by induction on $n$ the following statements. 

\medskip

{\bf$(\textrm{A}_n)$} Let $R$ be a positive real number, $X_1,\dots,X_s$ finitely many subanalytic subsets
of $\C_n(R)$, and  $\xi_1,\dots,\xi_l$ some bounded subanalytic nonnegative functions on $\C_n(R)$.

For every positive small enough real number  $\ep$, there exists a
Lipschitz  $x_1$-preserving subanalytic  homeomorphism
$h:\C_n(R,\ep) \to \C_n(R,\ep) $   such that $h(\ep,x)=x$ for all $x\in B(0_{\R^{n-1}},R\ep)$, and satisfying
\begin{enumerate}[(1)]\item\label{item_triv}
$h(0* X_{j,\ep})=X_{j} \cap \{x\in \R^n :0< x_1\le \ep\} $, for all $j=1,\dots s$, where $ X_{j,\ep}=  X_j\cap \{x\in \R^n :x_1=\ep\}$.
 \item\label{item_h_bi_proof}     The natural retraction by deformation onto the origin $$r:[0,1]\times \C_n(R,\ep) \to \C_n(R,\ep)$$ defined by $$r(s,x):=h(sh^{-1}(x)),$$ is Lipschitz.
 Indeed, there is a constant $C$ such that  for every fixed $s\in [0,1]$, the retraction $r_s$, defined by $$x\mapsto r_s(x):=r(s,x),$$ is $Cs$-Lipschitz.
 \item\label{item_h_1_lips_proof}  For each $\eta>0$,
 the restriction of $h^{-1}$ to $\{x\in X:\eta \le |x|\le \ep\}$ is Lipschitz and, for each $s\in (0,1]$, the map  $$r_s^{-1}:\C_n(R,s\ep) \to \C_n(R,\ep) $$ is Lipschitz. 

\item \label{item_tilda_decreasing} There is a constant $C$ such
that  for all $x\in\C_n(R,\ep)$  and all $k \leq l$ we have:\begin{equation}\label{eq decroissance fn up to
contant}\xi_k\circ h(sx) \leq C\,\xi_k\circ h(x).\end{equation}

\item \label{item_tilda_decreasing_inv} For each $ \delta>0$ there is a positive constant $c_\delta$  such that we
have for all $x\in\C_n(R,\ep)\setminus \C_n(R,\delta)$ and all $k \leq l$:\begin{equation}\label{eq decroissance inverse}c_\delta\, \xi_k\circ h(x)\leq\xi_k\circ h(sx).\end{equation}
 \end{enumerate}

\medskip

Before proving these statements, let us make it clear that these
yield the theorem. Let $X\subset \R^n$ be a subanalytic set. We can assume that $0\in cl(X)$ and  work nearby the origin. 
 The set  $$\check{X}:=\{(t,x) \in \R \times X: t=|x|\},$$ is a subset of $\C_{n+1}(R)$ (for $R>1$) to
which we can apply {\bf$(\textrm{A}_{n+1})$}. This provides a
Lipschitz  $x_1$-preserving homeomorphism
$$h: \C_{n+1}(R,\ep)\to \C_{n+1}(R,\ep) ,$$ which, thanks to (\ref{item_triv}) of {\bf$(\textrm{A}_{n+1})$},  gives rise to a homeomorphism $$H: 0*(S(0,\ep)\cap X)\to X \cap B(0,\ep) $$ (since $\check{X}_\ep=S(0,\ep)\cap X$).  
 Because the projection defined by $P(t,x):=x$ induces a  bi-Lipschitz  homeomorphism between
$\check{X}$ and $X$, properties (\ref{item_H_bi}) and (\ref{item_r_bi}) of
the theorem  come down from 
 $(\ref{item_h_bi_proof})$ and $(\ref{item_h_1_lips_proof})$  of {\bf$(\textrm{A}_{n+1})$}.

The assertions (\ref{item_tilda_decreasing}) and (\ref{item_tilda_decreasing_inv}) are not necessary to
derive statement of the theorem. They are required so as to
perform the proofs of $(\ref{item_h_bi_proof})$ and $(\ref{item_h_1_lips_proof})$ during the
induction step of the proof of {\bf$(\textrm{A}_{n})$}.

\medskip

As {\bf$(\textrm{A}_1)$} is trivial 
($h$ being the identity map), we fix some $n>1$. We also fix some globally subanalytic  subsets $X_1,\dots,X_s$ of $\C_n(R)$, $R>0$,  as well as
some globally subanalytic  bounded functions $\xi_1,
\dots,\xi_l:\C_n(R)\to\R$.

 \medskip




The induction hypothesis will be applied to the elements of a suitable cell decomposition  of $\R^{n-1}$. This requires some preliminaries.

Apply Lemma \ref{prop proj reg} to the collection of globally subanalytic sets constituted by the  (germs of) the
$X_i$'s, the set $\C_n(R)$ itself, and the union of the zero loci of the $\xi_i$'s together with the finite family of functions $\xi_1,\dots,\xi_l$.  We get a (germ of) $x_1$-preserving globally subanalytic  bi-Lipschitz map  $\Phi:\C_n(R)\to \C_n(R)$ and a
cell decomposition $\D$ of $\R^n$ such that $(\ref{item_compatible})$, $(\ref{item_graphe})$, and $(\ref{item_eq})$ of the
latter lemma hold.

Let $\Theta$  be a  cell of $\D$ which lies in $\Phi(\C_n (R))$ and which is a graph, say of a
 function $\eta:\Theta' \to \R$, $\Theta'$ standing for the basis of $\Theta$.  By $(\ref{item_graphe})$ of Lemma
\ref{prop proj reg} (and Remark \ref{rem_bded_der_lips}), $\eta$ is  a  Lipschitz
function. It thus may be extended to a globally subanalytic
Lipschitz function on the whole of $\R^{n-1}$. Repeating this for all the
cells $\Theta$ of  $\D$  lying in $\Phi(\C_n (R))$ which are graphs (i.e., not bands), 
we get finitely many globally subanalytic  Lipschitz functions $\eta_1,\dots,\eta_v$. Using the $\min$ and $\max$ operators if necessary, we may transform this family into a family satisfying $\eta_1\leq \dots
\leq \eta_v$.

 As we may work up to a $x_1$-preserving
 bi-Lipschitz map, in what follows we will identify $\Phi$ with the identity map. Hence, thanks to $(\ref{item_eq})$ of Lemma
\ref{prop proj reg}, we
may assume that on each cell every function $\xi_i $
 is $\sim$ to a function of the form which appears in (\ref{eq prep}).

We are now going to introduce some  bounded $(n-1)$-variable
functions $\sigma_1,\dots,\sigma_p$, to which we will apply (\ref{item_tilda_decreasing}) and (\ref{item_tilda_decreasing_inv}) of the induction hypothesis. This will be of service to establish $(2)$ and $(3)$ and to show   that  the $\xi_k$'s satisfy (\ref{item_tilda_decreasing}) and (\ref{item_tilda_decreasing_inv}). These $(n-1)$-variable functions will be provided by the estimate (\ref{eq prep}) and the $\eta_j$'s.

Fix an integer $1 \leq j < v $ and a cell $\Lambda$ of $\D$. Set for simplicity $D:=\Lambda \cap  (\eta_j,\eta_{j+1})$. 

For each $k\le l$, the function $\xi_k $ 
 is $\sim$ to a function like in (\ref{eq prep}) on $D$,  i.e., there exist $(n-1)$-variable  functions on $\pi_{e_n}(D)$ ($\pi_{e_n}$ denoting the orthogonal projection along $e_n$, the last vector of the canonical basis), say $\theta_k$ and $a_k$, and  $\alpha_k\in \Q$ such that: $$\xi_k(\xt,x_n) \sim |x_n-\theta_k(\xt)|^{\alpha_k}
a_k(\xt),$$ for  $(\xt,x_n)\in D \subset \R^{n-1}\times \R$.

As the zero loci of the $\xi_k$'s are included in the graphs of the $\eta_i$'s, we
have on  $\pi_{e_n}(D)$ for every $k$, either
  $\theta_k \leq \eta_j$ or $\theta_k \geq \eta_{j+1}$. We will assume for simplicity that $\theta_k \leq \eta_j$.

This means that for $(\xt,x_n)\in D \subset \R^{n-1}\times \R$:
\begin{equation}\label{eq min}\xi_k(\xt,x_n) \sim \min ((x_n-\eta_j(\xt)) ^{\alpha_k} a_k(\xt)
,(\eta_j(\xt)-\theta_k(\xt)) ^{\alpha_k}  a_k(\xt)),\end{equation}  if
$\alpha_k$ is negative and
\begin{equation}\label{eq max}\xi_k(\xt,x_n) \sim \max ((x_n-\eta_j(\xt))  ^{\alpha_k} a_k(\xt)
,(\eta_j(\xt)-\theta_k(\xt)) ^{\alpha_k}  a_k(\xt)),\end{equation} in the
case where $\alpha_k$ is nonnegative.

 First, consider the following functions:
\begin{equation}\label{eqdefkappak}\kappa_k(\xt):=(\eta_j(\xt)-\theta_{k}(\xt))^{\alpha_k} a_k(\xt), \qquad k=1,\dots ,l.\end{equation}
For every $k$,  the function $\kappa_k$  is bounded on $D$ since it is
equivalent to the function $\xt \mapsto \lim_{x_n\to \eta_j(\xt)} \xi_k(\xt,x_n)$ which is bounded. 
As the function $(\eta_{j+1}-\eta_j)$ is Lipschitz and vanishes at the origin (since it extends a function whose  graph lies in $\C_{n}(R)$) the function $x=(x_1,\dots,x_{n-1})\mapsto \frac{(\eta_{j+1}-\eta_j)(x)}{x_1}$  is bounded on $\C_{n-1}(R)$.  We thus can complete the family $\kappa$ by adding the function $\frac{(\eta_{j+1}-\eta_j)(x)}{x_1}$  as well as the functions  $\min((\eta_{j+1}-\eta_j)^{\alpha_k}a_k,1)$, $1\le k\le l$. This family $\kappa$ of course depends on the fixed set $D$ (intersection  of some cell $\Lambda$ of $\D$ with  $(\eta_j,\eta_{j+1})$, for some $j$). The union of all these families obtained for every  such set $D$  eventually provides us the desired
 collection of  functions $\sigma_1,\dots,\sigma_p$.

We now turn to the construction of the desired homeomorphism.
 Refine the cell decomposition $\pi_{e_n}(\D)$  into a cell decomposition $\D'$ of $\R^{n-1}$ compatible with the  zero loci of the functions $(\eta_j-\eta_{j+1})$. Apply the induction hypothesis to the family constituted by the
cells of $\D'$. This
provides a $x_1$-preserving homeomorphism   $\hti:\C_{n-1}(R,\ep) \to
\C_{n-1}(R,\ep),$
satisfying $\hti(\ep,x)=x$, for all $x\in B(0_{\R^{n-2}},R\ep)$.
 In addition,  thanks to 
the induction hypothesis, we may assume that the functions
$\sigma_1,\dots,\sigma_p$ as well as all  the functions $\xi_i(x,\eta_j(x))$ satisfy (\ref{eq decroissance fn up to
contant}) and  (\ref{eq decroissance inverse}) (with respect to $\hti$).

\medskip

 For
$q \in (\eta_j,\eta_{j+1})$, decomposed as $(\tilde{q},q_n) \in \R^{n-1}\times \R$ we set  
\begin{equation}\label{eq_nu}\nu(q):= \frac{q_n-\eta_{j}(\tilde{q})}{\eta_{j+1}(\tilde{q})-\eta_{j}(\tilde{q})}.\end{equation}
In order to define the desired homeomorphism $h$, take now an element $x\in\C_{n}(R,\ep) $. If the point $q:=\frac{\ep}{x_1}x$ belongs to $(\eta_{j},\eta_{j+1})$, for some $j<v$, we set:
$$h(x)=h(\xt,x_n):=(\hti(\xt),\nu(q)(\eta_{j+1}-\eta_j)\circ \hti(\xt) +\eta_j\circ \hti(\xt)).$$ If the point $q$ belongs to the graph of $\eta_j$, for some $j\le v$,  we set  $$h(x):=(\hti(\xt),\eta_j(\hti(\xt)).$$
 Since $\D'$ is compatible with the sets $\{ \eta_j=\eta_{j+1}\}$ and since $(1)$ holds for $\tilde{h}$ for each cell of $\D'$, the mapping $h$ is a homeomorphism on 
$\C_n(R,\ep)$.
 Moreover, as by construction $h$ satisfies $(1)$ for the  $\Gamma_{\eta_j}$'s,  this property holds true for all the cells of $\D$. As $\D$ is compatible with the $X_j$'s, this yields $(1)$ for $h$.

Let us check that (\ref{item_h_bi_proof}) and (\ref{item_h_1_lips_proof}) hold for $\tilde{h}$. Fix for this purpose $j <v$ and set for simplicity $\overline{\eta}_j:= \eta_{j+1}-\eta_j$.  In virtue of  the induction hypothesis, inequality (\ref{eq
 decroissance fn up to contant}) is fulfilled by the functions
 $\frac{\etab_j(\xt)}{\xt_1}$, where  $\xt=(\xt_1,\dots,\xt_{n-1})\in \C_{n-1}(R,\ep)$. It means that  (applying (\ref{eq decroissance fn up to contant})  with $s=\frac{\xt_1}{\ep}$) there is a constant $C$ such that  for all such $\xt$: 
\begin{equation}\label{eq_eta_j_H_R}
 \etab_j\circ \hti(\xt) \le \frac{C \xt_1}{\ep} \cdot \etab_j\circ \hti(\frac{\ep}{\xt_1} \xt).
\end{equation}
  Therefore, as $\hti$ is Lipschitz, the mapping $h$, which maps linearly the vertical segment $[(\xt,\frac{\xt_1}{\ep} \cdot\eta_{j}\circ \hti(\frac{\ep}{x_1}\xt)),(\xt,\frac{\xt_1}{\ep} \cdot\eta_{j+1}\circ \hti(\frac{\ep}{x_1} \xt))]$  onto the segment $[(\hti(\xt),\eta_j\circ \hti(\xt)),(\hti(\xt),\eta_{j+1}\circ \hti(\xt))]$,  must be  Lipschitz as
  well. Moreover, for the same reason, (\ref{eq decroissance inverse}) entails  that  $h^{-1}$ is Lipschitz on $\C_n(R,\ep)\setminus \C_n(R,\delta)$ for every $\delta>0$.
  
 Furthermore, as  (\ref{eq decroissance fn up to contant}) holds for $\frac{\etab_j}{x_1}$,
  composing with $\hti^{-1}$ in this inequality, we get for $x\in \C_n(R,\ep)$ and $s\in [0,1]$:
 \begin{equation}\label{eq_eta_j_r_s}
 \etab_j(r_s(x)) \le Cs \cdot \etab_j(x),
\end{equation}
which yields that $r_s$ is $Cs$ Lipschitz for each $s$ (that $r$ is Lipschitz with respect to $s$ is clear since  $h$ is itself Lipschitz). Finally, using (\ref{eq decroissance inverse}) in exactly the same way, we can show that $r_s^{-1}$ is Lipschitz for every positive $s$. This yields (\ref{item_h_bi_proof}) and (\ref{item_h_1_lips_proof}).

\bigskip


It remains to establish (\ref{item_tilda_decreasing}) and (\ref{item_tilda_decreasing_inv}).  
The claimed estimates are clear on the graphs of the $\eta_{j}$'s
 for we have required the functions $\xi_k(x,\eta_j(x))$ to satisfy such inequalities when applying the induction hypothesis. Hence, let us fix  $ k \leq
l$, $j<v$, as well as  $D \in \D'$, and show that  
 (\ref{eq decroissance fn
up to contant}) and (\ref{eq decroissance inverse}) hold for $\xi_k$ on the set $(\eta_{j|D},\eta_{j+1|D})$.

   Observe  that, as  $\xi_k$ is  bounded,
it is enough to prove this for the function $\min(\xi_k,1)$, and,   by (\ref{eq min}) and (\ref{eq max}),
  it actually suffices to show that the functions $\min((x_n-\eta_j(\xt))^{\alpha_k} a_k(\xt),1)$ and  $\min(|\theta_k-\eta_j|(\xt) ^{\alpha_k} a_k(\xt),1)$ both admit such estimates.
 For the latter function, this follows from the
induction hypothesis since we have required the $\kappa_i$'s (see
(\ref{eqdefkappak}))  and $\hti$ to have this property. We thus only need
to deal with the function $\min((x_n-\eta_j(\xt))^{\alpha_k} a_k(\xt),1)$.

For simplicity, we set $$F(\xt,x_n):=(x_n-\eta_j(\xt))^{\alpha_k}\cdot a_k(\xt),$$
and $$G(\xt):=(\eta_{j+1}-\eta_j)(\xt)^{\alpha_k} \cdot a_k(\xt).$$

We have to  show the desired inequalities for $\min(F,1)$. By definition of $\nu$ (see (\ref{eq_nu})) we have for $x=(\xt,x_n)\in (\eta_{j|D},\eta_{j+1|D})$:
\begin{equation}\label{eq F G}F(x)=\nu(x)^{\alpha_k} \cdot G(\xt).\end{equation}

 Remark that the function $\nu(h(sx))$ is
constant with respect to $s\in [0,1]$, which
implies that for $x=(\xt,x_n)\in \C_n(R,\ep)$ and $s\in [0,1]$ we have:
\begin{equation}\label{eq2 F G}F(h(sx))=\nu(h( x))^{\alpha_k} \cdot G(\hti(s\xt)).\end{equation}

We assume first that $\alpha_k$ is negative. Thanks to the
induction hypothesis, we know that for $0< \delta \leq s\le 1$:  $$c_\delta  \min (G\circ \hti(x),1)\leq \min(G\circ \hti(s\xt),1)\leq C \min
(G\circ \hti(\xt),1),$$
for some positive constants $c_\delta,C$ (with $C$ independent of $\delta$). 
 Multiplying by $\nu^{\alpha_k}$ and applying
 (\ref{eq2 F G}), this implies that:
$$c_\delta\min(F\circ h(x),1,\nu(h(x))^{\alpha_k})\ \leq  \min(F\circ h(sx),1,\nu(h(x))^{\alpha_k})\leq C \min (F\circ h(x),1,\nu(h(x))^{\alpha_k}) $$
But, as $\alpha_k$ is negative,
$\min(F,\nu^{\alpha_k},1)=\min(F,1)$ and  we are done.

We now assume that $\alpha_k$ is nonnegative. This implies that $F $ is  bounded  (since, by (\ref{eq max}), $F\le \xi_k$), which entails that $G$ is bounded as well.  Moreover, thanks to (\ref{eq2 F G}), it actually suffices  to show the desired inequality for $G$ and $\hti$. As $G$ is bounded, $G\sim \min (G,1)$, which  satisfies (\ref{eq decroissance fn up to contant})  and (\ref{eq decroissance inverse}), in virtue of the induction hypothesis ($\min (G,1)$ is one of the $\sigma_i$'s). This
establishes (\ref{item_tilda_decreasing}) and (\ref{item_tilda_decreasing_inv}).
\end{proof}
\begin{rem}\label{rem_conical_same}
 We have proved that, given finitely many subanalytic set germs $X_1,\dots,X_s$ at $x_0\in\R^n$, the respective  homeomorphisms of the Lipschitz conic structure of the $X_i$'s can be required to be induced by the same Lipschitz homeomorphism $H:x_0*S(x_0,\ep)\to \overline{B}(x_0,\ep)$.
\end{rem}
\end{subsection}
 \begin{subsection}{horizontally $C^1$ mappings.}\label{sect_hor_C1} We wish to prove that globally subanalytic bi-Lipschitz homeomorphisms pull-back  weakly differentiable forms to weakly differentiable forms (see section \ref{sect_lip_lp}). This requires stratification theory and we shall make use of  the notion of horizontally $C^1$ mappings introduced in   \cite{mt}. In this section, we give basic definitions and prove a preliminary lemma. 

\begin{dfn}
A {\it stratified mapping} is the data of a mapping $h:X\to Y$, $X\subset \R^n$, $Y\subset \R^k$, together with some stratifications $\Sigma$ and $\Sigma'$ of $X$ and $Y$ respectively, such that $h$ smoothly maps every stratum of $\Sigma$ into a stratum of $\Sigma'$. 

A stratified mapping  $h:(X,\Sigma)
\to (Y,\Sigma')$  is said to be {\it horizontally
$C^1$}  \index{horizontally $C^1$} if, for any sequence
$(x_l)_{l\in \N}$ in a stratum $S$ of $\Sigma$ tending to some
point $x$ in a stratum $S'\in \Sigma$ and for any  sequence $u_l \in
T_{x_l} S$ tending to a vector $u$ in $T_x S'$, we have
$$\lim d_{x_l} h_{|S}(u_l)=d_x h_{|S'} (u).$$
\end{dfn}

If $h$ is horizontally $C^1$ then the derivative of $h_{|S}$ is bounded away from infinity on every bounded subset of $S$ for every $S \in \Sigma$.
  The lemma below can be considered as a converse of this observation. It relies on the existence of Whitney $(a)$ regular stratifications which is well-known for subanalytic mappings \cite{tesc, shiota}. We recall that a stratification is {\it Whitney $(a)$  regular} if for every sequence $x_i$ in a stratum $S$ converging to a point $y$ in a stratum  $S'$, in such a way that $T_{x_i} S$ has a limit in the Grassmannian, we have $\lim T_{x_i} S\supset T_y S'$.


\begin{lem}\label{lem_h_hor_C1}
Let $h:X\to Y$ be a subanalytic continuous mapping. If $|d_x h|$ (which exists almost everywhere) is bounded on every bounded subset of $X$ then there
 exist two stratifications $\Sigma$ and $\Sigma'$, of $X$ and $Y$ respectively, making of $h:(X,\Sigma)\to (Y,\Sigma')$  a horizontally $C^1$ stratified mapping.
\end{lem}
\begin{proof}Let $\pi_1:\Gamma_h \to X$ (resp. $\pi_2:\Gamma_h\to Y$) be the projection onto
the source (resp. target) space of $h$.  Take Whitney $(a)$ regular stratifications $\Sigma_h$ and $\Sigma'$ of $\Gamma_h$ and $Y$ respectively such that  $\pi_2:(\Gamma_h,\Sigma_h) \to (Y, \Sigma')$ is a stratified mapping. 

Let $\Sigma$ be the stratification of $X$ constituted by the respective images of the strata of $\Sigma_h$
under the mapping $\pi_1$. 
 Observe that  $h:(X,\Sigma)\to (Y,\Sigma')$ is a stratified mapping. 


 In order to show that $h$ is horizontally $C^1$ (with respect to $\Sigma$), fix a stratum $S$ of $\Sigma$, a
sequence $x_l \in S$ tending to a point $x$ belonging to a stratum $S'\in \Sigma$, as well as
a sequence $u_l \in T_{x_l} S$ tending to some $u \in T_x S'$. 

Let $Z$ be the stratum of $\Sigma_h$ that projects onto $S$ via $\pi_1$ and   set  (extracting a subsequence if necessary, we may assume that this sequence is convergent) $$\tau:=\lim T_{(x_l,h(x_l))}  Z.$$  

\noindent {\bf Claim.} The restriction of $\pi_1$ to $\tau$ is one-to-one. 

To see  this, observe that, as $(\Gamma_h)_{reg}$ is dense in $\Gamma_h$, for every $l$ we can find an element $y_l\in (\Gamma_h)_{reg}$ close to $(x_l,h(x_l))$.  For every $l$, let $Z^{l}$ be the stratum of $\Sigma_h$ containing $y_l$  (choosing $y_l$ sufficiently generic, we may assume that $Z^l$ is open in $\Gamma_h$).  By the Whitney $(a)$ condition, the angle between $T_{(x_l,h(x_l))} Z$ and $T_{y_l} Z^{l}$ is small if $y_l$ is chosen sufficiently close from $(x_l,h(x_l))$. It means that we can assume that $T_{y_l} Z^{l}$ tends to a limit $\tau'$ which contains $\tau$. As $h$ has locally bounded derivative $\pi_{1|\tau'}$ must be one-to-one (the graph of a mapping having a bounded first order derivative may not have a
vertical limit tangent vector), yielding the claim.

For every $l$, there
is a unique vector $v_l \in T_{(x_l,h(x_l))} Z$ which projects
onto $u_l$. The above claim implies that the norm of $v_l$ is bounded (since otherwise we would have $\lim \pi_1(\frac{v_l}{|v_l|})=0$)
and we may assume that $v_l$ is converging to a vector $v$. The
vector $v$ then necessarily projects onto $u$.

 Let $Z'$  be the stratum of $\Sigma_h$
that projects onto $S'$ and let $w$ be the vector tangent to $Z'$ at $(x,h(x))$ which projects onto $u$. By the Whitney $(a)$ condition $(v-w)\in\tau$.  Therefore, since  $(w-v)$
lies in the kernel of $\pi_{1|\tau}$, it must be zero (by the above claim). Hence, $v$ is tangent to  $Z'$, which entails that:
$$\lim d_{x_l} h_{|S} (u_l)=\lim \pi_2(v_l)=\pi_2(v)= d_{x} h_{|S'}(u).$$
\end{proof}

\begin{rem}\label{rem_hor_C1} In the situation of the above lemma, take in addition finitely many subanalytic subsets $A_1,\dots,A_k$ of $X$. 
 Since the stratification of $X$  given by the above lemma is provided by a Whitney stratification of the graph of $h$, we see that we can require the $A_i$'s to be unions of strata of  this stratification.
\end{rem}
\end{subsection}
\begin{subsection}{Stratified forms.}\label{sect_strat}    We recall the definition of stratified forms and then prove that these forms naturally give rise to weakly differentiable forms. 

\begin{dfn}\label{dfn_stratified_form}
Let $X\subset \R^n$ be subanalytic and let $\Sigma$ be a stratification of $X$. 

A {\it stratified differential $0$-form on $(X,\Sigma)$} is a collection of functions $\omega_S:S\to \R$, $S \in \Sigma$, that glue together into a continuous function on $X$.  

A {\it stratified differential $j$-form on $(X,\Sigma)$}, $j>0$, is a collection $(\omega_S)_{S \in \Sigma}$  where, for every $S$, $\omega_S$ is a continuous differential $j$-form on $S$ such that for any sequence $(x_i ,\xi_i)$, $i \in \N$, with $x_i$ tending to some $x\in S'\in \Sigma$ and $\xi_i$ tending to some $\xi \in \otimes^j  T_xS'$ we have $$\lim \omega_S(x_i,\xi_i)=\omega_{S'}(x,\xi).$$ 

 We say that $\omega=(\omega_S)_{S \in \Sigma}$ is {\it differentiable}  if $\omega_S$ is $C^1$ for every $S\in \Sigma$ and if $d\omega:=(d\omega_S)_{S\in \Sigma}$ is a stratified form.
\end{dfn}
 The integral of a stratified form on $X$ is then defined as the sum of the respective integrals of the corresponding forms on the top dimensional strata (see \cite{stokes} for details). An interesting feature of stratified forms is to admit a Stokes' formula, which we recall now.

\begin{pro}\label{cor_stokes_C0_manifolds} \cite{stokes}
Let $M\subset \R^n$ be a subanalytic $C^0$ compact oriented $m$-dimensional manifold with boundary $\pa M$ and let $\Sigma$ be a stratification of $M$.   For any differentiable stratified $m$-form  $\omega:=(\omega_S )_{S\in \Sigma}$ on $M$ we have:
$$\int_{M} d\omega=\int_{\pa M} \omega.$$
\end{pro}

Using this result, we can show that stratified forms naturally give rise to weakly differentiable forms:

\begin{lem}\label{lem_stratified_are_weakly_diff}
  Let $\alpha:=(\alpha_S)_{S\in \Sigma}$ be a stratified $j$-form, where $\Sigma$ is a stratification of a smooth submanifold $M$ of $\R^n$. If $\alpha$ is differentiable then it is weakly differentiable in the sense that the form $\alpha'$ defined by  $\alpha'(x):=\alpha_S(x)$, for $S\in \Sigma$, $\dim S=\dim M$, and   $x \in S$ (this form is defined almost everywhere) is weakly differentiable.
\end{lem}
\begin{proof}  As the problem is local (we can use a partition of unity), we can assume that $M$ is a small open ball $B(\orm,\ep)$, and consequently that $cl(M)$ is a manifold with boundary $S(\orm,\ep)$.
 We are going to show that $\alpha'$ is weakly differentiable and that its weak exterior differential is the form $\alpha''$ defined (almost everywhere) by $\alpha''(x):= d \alpha_S (x)$, for $x\in  S \in \Sigma$ and $\dim S=m:=\dim M$. Let $\varphi\in \Lambda_{or} ^{m-j-1} (M)$ and observe that the form $\beta:=(\beta_S)_{S\in \Sigma}$ defined by $\beta_S(x):= \alpha_S(x)\wedge \varphi(x)$ is a differentiable stratified form and that its exterior differential is the stratified form $(d \alpha_S \wedge \varphi+(-1)^{j}\alpha_S\wedge d\varphi )_{S\in \Sigma}$.   It thus suffices to establish that $\int_M d\beta =0.  $

 Take a stratification $\Sigma'$ of $cl(M)$ such that  all the strata of $\Sigma$ are unions of strata of $\Sigma'$ and such that $S(\orm,\ep)$ is a union of strata. As $\varphi$ has compact support in $M$, the form $\beta$ gives rise to a stratified form on $\Sigma'$   which is identically zero on the strata of $S(\orm,\ep)$. The required equality then follows from Proposition \ref{cor_stokes_C0_manifolds}. 
\end{proof}
\end{subsection}

\begin{subsection}{Weakly smooth forms and subanalytic bi-Lipschitz maps.}\label{sect_lip_lp}Combining the results of sections \ref{sect_hor_C1} and \ref{sect_strat}, we now can prove that bi-Lipschitz subanalytic mappings (not necessarily smooth) naturally induce isomorphisms in $L^p$ cohomology. Let us here emphasize that, although differentiable forms are not required to be subanalytic, the subanalytic character of the mappings is essential. Let $M$ be a subanalytic smooth submanifold of $\R^n$.     
\begin{pro}\label{pro_pullback_weakly differentiable forms}
Let $h:M\to M'$  be a 
Lipschitz  subanalytic mapping, with   $M'\subset \R^k$ smooth submanifold.  For every smooth form $\omega$  on $M'$, the form $h^*\omega$ (which is well defined almost everywhere) is weakly differentiable and satisfies
$\overline{d} h^*\omega=h^* d \omega$, almost
everywhere.

Moreover, if $h$ is locally bi-Lipschitz then the same conclusion holds  for every weakly differentiable form $\omega$ on $M'$.
\end{pro}
\begin{proof}
By Lemma \ref{lem_h_hor_C1}, $h$ is horizontally $C^1$ with
respect to some stratifications of $M$ and $M'$. For every $\omega\in \Omega^j(M')$, $h^*\omega$ thus gives rise to a differentiable stratified form on this stratification. By Lemma \ref{lem_stratified_are_weakly_diff},  this means that $h^*\omega$ is weakly differentiable, and thanks to Stokes' formula for stratified forms (Proposition \ref{cor_stokes_C0_manifolds}), the formula $\overline{d} h^*\omega=h^* d \omega$  easily follows by integration by parts.
 
  To prove the last statement, fix a form $\omega\in \ws^j(M')$, $j\in \N$ as well as an  open subset $U$ of $M'$ on which $\omega$ is $L^1$. 
As  $\Omega^j_{1}(U)$ is dense  in $\ws^j_{1}(U)$, we can find a sequence $\omega_i \in \Omega_{1}^j(U)$ such that   $\omega_i\to \omega$ and $d\omega_i\to \db \omega$ for the $L^1$ norm. If $h$ is bi-Lipschitz,  this implies that $h^*\omega_i$ tends to $h^*\omega$ and that $h^*d \omega_i$ tends to $h^*\db \omega$ for this norm. Moreover, since we have proved that the result holds true in the smooth case, we also know that for every $\varphi\in \Lambda_{or} ^{m-j-1}(U)$
$$\int_U h^*\omega_i \wedge d\varphi =(-1)^{j+1}\int_U d \omega_i\wedge \varphi  .$$ 
 Passing to the limit as $i\to \infty$, we get the desired result.
\end{proof}

This proposition leads us to the subanalytic bi-Lipschitz invariance of $L^p$ cohomology:

\begin{pro}\label{pro lp bi-lips invariant}
Let  $j\in \N$ and $p\in [1,\infty]$.  If $h:M\to M'$  is a 
  subanalytic bi-Lipschitz mapping, with $M'\subset \R^{n'}$ smooth submanifold, then $\hp ^j (M)\simeq \hp^j(M')$. 
\end{pro}
\begin{proof}
By  the preceding proposition and Corollary \ref{cor_ws_isom_smooth},  $h$ induces an isomorphism between   $H_{p} ^j (M)$ and $H_{p}^j(M')$ for all 
$j$ and all $p$.
\end{proof}
Similarly,  subanalytic Lipschitz homotopies (not necessarily differentiable) induce operators on smooth forms exactly like in the case of smooth homotopies.
\begin{pro}\label{pro_homotopy}
Let $h:[0,1]\times M \to M$ be a  subanalytic Lipschitz homotopy and let $\pa_t$ denote the constant vector field $(1,0)$ on $[0,1]\times M$. If we set  for $x\in M$ and  $\omega\in \Omega^j(M)$:
 $$\hn\omega(x):=\int_0 ^1 (h^*\omega)_{\pa_t}(t,x)dt,$$
then   we have
 \begin{equation}\label{eq_chain_homotopy}
  \db \hn \omega+\hn d\omega=h_1^* \omega -h_0 ^*\omega,
 \end{equation}
where $h_i:M\to M$, $i=0,1$, is defined by $h_i(x)=h(i,x)$. 
\end{pro}
\begin{proof}
Proposition \ref{pro_pullback_weakly differentiable forms} implies that $h^*\omega$ is weakly differentiable and that $\db h^*\omega=h^*d\omega$. Moreover,  thanks to Lemma \ref{lem_h_hor_C1},  $h^*\omega$ gives rise to a stratified form. 
 
 Thanks to our Stokes' formula for  stratified forms, we now can end the proof with an integration by parts. Namely, if $\varphi \in \Lambda^{m-j} _{or} (M)$,  regarding it as a form on $[0,1]\times M $ constant with respect to $t$, we can write (for relevant orientations):
\begin{eqnarray}\label{eq_ipp}
\int_{M} \K\omega \wedge d\varphi&=&\int_{ [0,1]\times M
} h^*\omega \wedge d\varphi
\nonumber\\&=&(-1)^{j} \int_{[0,1]\times M
}\overline{d}(
h^*\omega \wedge \varphi)- (-1)^{j}\int_{[0,1]\times M
} \overline{d}h^*\omega \wedge\varphi.\end{eqnarray}
 By Stokes' formula for stratified forms,  we also have:
$$\int_{[0,1]\times M
}\overline{d}(
h^*\omega \wedge \varphi) =\int_M h_1^* \omega\wedge \varphi- \int_M h_0^* \omega\wedge \varphi. $$
Together with (\ref{eq_ipp}), this yields the desired equality.
\end{proof}
\end{subsection}

\end{section}

\begin{section}{Some operators on $L^p$ forms}\label{sect_hom_hop}
We are going to define some operators on  $L^p$ forms on subanalytic varieties which will be useful to establish our Poincar\'e Lemma for  $L^p$ cohomology  (Lemma \ref{lem x normal Poincare lemma}). The usual Poincar\'e Lemma is devoted to smooth forms on an open ball or more generally on the so called star-shaped domains.
On this kind of domains, it is well-known that some retractions by deformation give rise to differential operators on forms. The  local conic structure given in Theorem \ref{thm_local_conic_structure} will make it possible to define some operators by the same process as on the star-shaped domains. We start by defining them and then study their properties. 

We fix for all this section an $m$-dimensional subanalytic submanifold $M$ of $\R^n$.
Set $X:=cl(M)$, fix $x_0\in X$,  
and apply Theorem \ref{thm_local_conic_structure}  to the germ of $X$ at $x_0$. This provides a positive real number $\ep$ as well as a Lipschitz  subanalytic homeomorphism
  $$H: x_0* (S(x_0,\ep)\cap X)\to  \Bb(x_0,\ep) \cap X,$$
preserving the distance to $x_0$ and satisfying conditions $(i)$ and $(ii)$ of the latter theorem.

   For simplicity, we then set for  this section $$N_{x_0}=S(x_0,\ep)\cap M$$ and $$U_{x_0}=B(x_0,\ep)\cap M.$$
   We can assume (see Remark \ref{rem_conical_same}) that $H$ maps the open cone $(x_0*N_{x_0})\setminus N_{x_0}\cup\{x_0\}$ onto $U_{x_0}$.
  In particular, $H$ gives rise to a  globally subanalytic homeomorphism:
  \begin{equation}\label{eq h pour K_nu}
h:(0,1) \times N_{x_0} \to U_{x_0}, \qquad (t,x)\mapsto h(t,x):=H(tx).\end{equation}
   For simplicity, we also set 
   $$\zx=(0,1)\times N_{x_0}.$$
   
%

\begin{subsection}{The operator $\K_\nu, \, \nu>0$} 
Denote by $\pa_t $ the constant vector field $(1,0)$   on $(0,1)\times N_{x_0}$ and   fix an $L^1$ differential $j$-form $\omega$ on $\ux$ with $j\ge 1$.
 We first define a differential form $\hn_\nu \omega$ by setting  for almost every $(t,y) \in (0,1)\times N_{x_0}$ and $0<\nu\leq 1$:
 \begin{equation}\label{eq_def_alpha}\hn_\nu\omega(t,y)=\int_\nu ^t  (h^*\omega)_{\pa_t} (s,y) ds.\end{equation} The desired operator is then defined by pushing forward this differential form by means of $h$: 
 \begin{equation}\label{eq_def_K}
  \K_\nu \omega=h^{-1*}\hn_\nu \omega.
 \end{equation}
 

This defines an operator $\K_\nu$ on $L^1$ differential forms for every $\nu \in (0,1]$. 

\subsection{The operator $\K_0$} The case $\nu=0$ is more delicate since we are not sure that the mapping  $t \mapsto (h^*\omega)_{\pa_t} (t,y)$ is $L^1$ on $[0,1]$. This fact is however clearly true if $\omega$ is an $L^\infty$ form.
The proposition below shows that we actually can  define $\K_0 \omega$ analogously when $\omega$ is an $L^p$ form with $p$ sufficiently big.  
\begin{prodfn}\label{pro_dfn_K_0}
 For $p\in [1,\infty]$ sufficiently big, the form $ (h^*\omega)_{\pa_t} $ is $L^1$ on $(0,1)\times N_{x_0}$ for every $L^p$ differential $j$-form $\omega$, $j\ge 1$.
 
 For such $p$ and $\omega$, the differential form \begin{equation}\label{eq_h_0}\hn_0\omega(t,y):=\int_0 ^t  (h^*\omega)_{\pa_t} (s,y) ds\end{equation} is thus (almost everywhere on $(0,1)\times N_{x_0}$) well-defined and we can set  
 $$\K_0 \omega=h^{-1*}\hn_0 \omega.$$
\end{prodfn}
\begin{proof}
 %

  The function $(s,x)\mapsto \jac h(s,x)$ (this Jacobian is well defined  on a subanalytic dense subset of $\zx$) is globally subanalytic.   As $h$ is bi-Lipschitz above the complement  of every neighborhood of the origin, $\jac h(s,x)$ can only tend to zero when $s$ goes to zero.    Therefore, by \L ojasiewicz's inequality (see (\ref{eq_loj})), there is a positive integer $k$ and a constant $C$ such that for $(s,x)\in (0,1)\times N_{x_0}$
  \begin{equation}\label{eq_jac_h_s}s^k\le C \jac h(s,x).\end{equation}
           We are going to prove that $ (h^*\omega)_{\pa_t} $ is $L^1$ for all $L^p$ forms $\omega$ when $p>k+1$.   Fix such a form $\omega$ and such a real number $p$.                                         

          Since $h$ has bounded first derivative, it is enough to show that $\omega\circ h$ is $L^1$. For this purpose, let us notice that since $\omega$ is $L^p$, so is $|\omega\circ h|\cdot (\jac h)^{\jpp}$.  It thus suffices to show that $(\jac h)^{-\jpp}$ is $L^q$, where $q\ge 1$ is the H\"older conjugate of $p$.  To prove this, write
$$\int_{\zx} (\jac h)^{-\frac{q}{p}} \overset{(\ref{eq_jac_h_s})}{\lesssim}\int_{N_{x_0}}\int_0^1 s^{-\frac{kq}{p}}ds \lesssim \int_0^1 s^{-\frac{kq}{p}}ds  = \int_0 ^1 s^{\frac{k}{p-1}} ds <\infty,$$
since $k<p-1$. This establishes that $\omega\circ h$ is $L^1$, which  yields that so is $ (h^*\omega)_{\pa_t} $.
  \end{proof}
\end{subsection}
  \begin{subsection}{The homotopy $\rho_\nu$}\label{sect_r_s}
  Given $\nu \in [0,1]$, we can define a homotopy $\rho_\nu:(0,1]\times \ux \to \ux$ as follows. Let for $(t,s,y) \in (0,1)\times(0,1)\times \nx$  
  \begin{equation}\label{eq_theta_nu}
\theta_\nu(t,s,y):=h(ts+(1-t)\nu, y ).   
  \end{equation}
   We then push-forward $\theta_\nu$ by means of $h^{-1}$ by setting for $(t,x)\in (0,1]\times \ux$, $$\rho_\nu(t,x):=\theta_\nu (t, h^{-1}(x)).$$  
  
   As the homeomorphism $H$ (used at the beginning of this section to define $\K_\nu$) was assumed to send the open cone  $(x_0*N_{x_0})\setminus (N_{x_0}\cup \{x_0\})$ onto $\ux$, we see that this homotopy stays in  $U_{x_0}$ for all $t\in(0,1]$. Notice also that it follows from  Theorem \ref{thm_local_conic_structure} that this mapping is locally Lipschitz near every point of $(0,1]\times\ux$.
   
   Remark also that for every $x\in \ux$,  $\rho_0(t,x)$ coincides with $r_t(x)$, where $r$ is the mapping given in the latter theorem (although  the mapping $r$ is defined on $X$, we will regard it in the sequel as a mapping from $(0,1]\times \ux$ into $\ux$ and $r_t$ as a mapping from $\ux$ to itself for all $t\in (0,1]$).

  Let us here stress the fact that Theorem \ref{thm_local_conic_structure}  ensures that there is a constant $C$ such that for all $t\in (0,1]$, the mapping $r_t:\ux\to \ux$ is $Ct$-Lipschitz. Moreover, for every $t\in (0,1]$, the mapping $r_t$ is bi-Lipschitz.
  
  Note also that since $H$ preserves the distance to $x_0$, we have  for all $(t,x)\in (0,1] \times \ux$:
 \begin{equation}\label{eq_r_s dist to xo}
                                                                     |r_t(x)-x_0|=t|x-x_0|.
                                                                    \end{equation}
  

   The next proposition provides an alternative definition of the operator $\K_0$ using $r$. This kind of computation is of course very classical.  As $r$ is Lipschitz, this characterization will be helpful to estimate the $L^p$ norm of $\K_0\omega$ in section \ref{sect_lp_bounds}.

\begin{pro}\label{pro_K_0}  
For every $L^1$ form $\omega$ on $\ux$, we have for each $\nu\in (0,1]$:
  \begin{equation}\label{eq_K_nu_2eme_forme} \K_\nu \omega(x) = \int_0 ^1  (\rho_\nu^*\omega)_{\pa_t} (t,x)dt, \end{equation}
 where $\pa_t$ is the constant vector field $(1,0)$ on $[0,1]\times \ux$.
 
Moreover, if $r:(0,1)\times \ux \to \ux$ is the just above defined mapping, we have for each $p\in [1,\infty]$  large enough, each $L^p$ form $\omega$, and each $x\in \ux$:
  \begin{equation}\label{eq_K_0_2eme_forme}\K_0 \omega(x)= \int_0 ^1  (r^*\omega)_{\pa_t} (t,x)dt,  \end{equation}
                                                                  \end{pro}
           \begin{proof}Fix  $(t,y)\in \zx$ and  $\nu\in (0,1]$. Making the substitution $s=ut+(1-u)\nu$, $u\in [0,1]$, in the integral defining $\hn_\nu\omega(t,y)$ (equality (\ref{eq_def_alpha})) we obtain 
      $$ \hn_\nu \omega (t,y)     =t\int_0 ^1 (h^*\omega)_{\pa_t}(ut+(1-u)\nu,y) du=\int_0 ^1 (\theta_\nu^*\omega)_{\pa_u} (u,t,y) du, $$
            where $\pa_u$ is the constant vector field $(1,0,0)$ on $[0,1]^2\times \ux$ (and $\theta_\nu$ is as in (\ref{eq_theta_nu})).  As  $\K_\nu=h^{-1*} \hn_\nu $ and $\rho_\nu(u,x)=\theta_\nu(u, h^{-1}(x))$ for all $u$, after a pull-back of by means of $h^{-1}$, we get (\ref{eq_K_nu_2eme_forme}).

       Observe that if the necessary integrability conditions are satisfied then the  above computation applies in the case where $\nu=0$ as well. Hence,  since $\rho_0(t,x)=r(t,x)$,  this argument yields (\ref{eq_K_0_2eme_forme}) for all $p\in [1,\infty]$ sufficiently large for the conclusion of  Proposition \ref{pro_dfn_K_0}  to hold (i.e., for the form which is integrated in (\ref{eq_h_0}) to be $L^1$).
           \end{proof}
\end{subsection}

\begin{subsection}{$\K_\nu$ and  weakly differentiable forms}
 Given $\nu\in (0, 1)$, let   $\pi_\nu:=h\circ P_\nu \circ h^{-1}$, where $P_\nu (t,x):=(\nu ,x)$.    

%
\begin{pro}\label{pro_proprietes_de_K_nu}
 For all $\omega \in \Omega ^j_1 (\ux)$, $j\ge 1$, and all $\nu \in (0,1)$, the form $\K_\nu\omega$ is  weakly differentiable and we have: 
\begin{equation}\label{eq_K_homot_operator}\overline{d}\K_\nu\omega+\K_\nu d\omega=\omega-\pi_\nu^*\omega.\end{equation}
In particular,  if $\omega$ is equal to zero in the vicinity of $N_{x_0}$ then  we have:
\begin{equation}\label{eq_K_homot_operator_cpct_supp}\overline{d}\K_1\omega+\K_1 d\omega=\omega.\end{equation}
\end{pro}

\begin{proof}
Equality (\ref{eq_K_homot_operator}) follows from (\ref{eq_chain_homotopy}) and (\ref{eq_K_nu_2eme_forme}). The second statement follows from the fact that, if $\omega$ is equal to zero in the vicinity of $N_{x_0}$ then $\pi^*_\nu \omega$ vanishes and $\K_\nu\omega=\K_1\omega$,  for all $\nu$ close to $1$.
 \end{proof}

We wish to establish an analogous result  in the case $\nu=0$ for $p$ sufficiently large  (Proposition \ref{pro_proprietes_de_K_nu_Lp}). This is a bit more delicate since the forms are not defined at $x_0$. We shall need the following fact. 

 \begin{pro}\label{pro_Knu_phi_tend_vers_zero}
  For $p$ large enough we have for each $L^p$ $j$-form $\omega$ on $\ux$, $j\ge 1$, and each $\varphi\in \Lambda_{or} ^{m-j+1}(\ux)$
  $$\lim_{\nu \to 0} <\K_\nu\omega ,\varphi>=<\K_0\omega,\varphi>.$$ 
 \end{pro}
 \begin{proof}
Take $p$ large enough for  the conclusion of  Proposition \ref{pro_dfn_K_0} to hold and fix an  $L^p$ $j$-form $\omega$ on $\ux$, $j\ge 1$.
 Since $h$ is bi-Lipschitz on the preimage of the support of any compactly supported form $\varphi$, it suffices to establish that $\hn_\nu\omega$ tends to $\hn_0\omega$ for the $L^1$ norm.  Remark that $\hn_\nu\omega$ tends to $\hn_0\omega$ pointwise. As a matter of fact, since
$$|\hn_\nu \omega(t,x)|=|\int_\nu ^t (h^*\omega)_{\pa_t} (s,x)ds|\le \int_0 ^1 |(h^*\omega)_{\pa_t} (s,x)|ds $$
 which is $L^1$ on $\zx$  (and constant with respect to $t$), the result follows from Lebesgue's dominated convergence theorem.
 \end{proof}

\begin{pro}\label{pro_proprietes_de_K_nu_Lp}
For $p$ large enough and $j\ge 1$, we have  for all  $\omega \in \Omega_{p} ^j (\ux)$:
\begin{equation}\label{eq_K_homot_operator_Lp}\overline{d}\K_0\omega+\K_0 d\omega=\omega.\end{equation}
\end{pro}
\begin{proof}
For $\nu \in (0,1)$, define a mapping $h_\nu:N_{x_0}\to N_{x_0}$ by $h_\nu(x):=h(\nu,x)$.  If $\nu$ remains bounded below away from zero then so does the function $(\nu,x)\mapsto\jac h_\nu(x)$. Consequently,
by \L ojasiewicz's inequality (see \ref{eq_loj}), there exists a  rational number $k$ such that (almost everywhere) on $ (0,1)\times N_{x_0}$ we have:
 \begin{equation}\label{eq_jac_h_nu}
 \nu^k \lesssim  \jac h_\nu(x).
 \end{equation}
 Fix $p\ge k+1$ sufficiently large  for the conclusion of Proposition \ref{pro_Knu_phi_tend_vers_zero} to hold and take a differential form   $\omega \in \Omega_{p} ^j (\ux)$.  We have to prove that for all $\varphi \in \Lambda_{or}^{m-j}(\ux)$:
 \begin{equation}\label{eq_proof_K_0_ws}
 <\K_0\omega,d\varphi>=<\omega,\varphi> - < \K_0 d\omega,\varphi>.
 \end{equation}

  Fix such a differential form $\varphi$. As $\omega$ and $\db \omega$ are  $L^1$,      by Proposition \ref{pro_proprietes_de_K_nu}, we know that  for all $\nu \in (0,1)$
 $$<\K_\nu\omega,d\varphi>=<\omega,\varphi>-<\pi_\nu ^*\omega,\varphi> - < \K_\nu d\omega,\varphi>.$$
  Moreover, applying Proposition \ref{pro_Knu_phi_tend_vers_zero}  to both $\omega$ and $d \omega$ we see that 
$$\lim_{\nu \to 0} <\K_\nu \omega,\varphi>=<\K_0 \omega,\varphi> \;\mbox{ and }\; \lim_{\nu \to 0} <\K_\nu d\omega,\varphi>=<\K_0d \omega,\varphi>.$$
As a matter of fact, (\ref{eq_proof_K_0_ws}) reduces to show that there is a sequence $\nu_i$ tending to zero such that 
 \begin{equation}\label{eq_pi_nu_tend_vers_zero}
  \lim_{i \to +\infty} <\pi_{\nu_i}^*\omega,\varphi>=0.
 \end{equation}
 For simplicity,   set
 $$\theta(\nu):=\int_{z\in N_{x_0}} |\omega(h_\nu(z))|^p .$$
 Observe first that by definition of $\theta$ we have
 $$\int_0 ^1  \nu^k \theta(\nu)d\nu \overset{(\ref{eq_jac_h_nu})}{\lesssim} \int_0 ^1 \int_{z\in N_{x_0}} |\omega(h_\nu(z))|^p \jac h_\nu(z) d\nu =|\omega|_p <\infty, $$
 which means that the function $\nu^k\theta(\nu)$
  belongs to $L^1((0,1))$. Since $p\ge k+1$, this implies that there exists a sequence of positive numbers $\nu_i$ tending to zero such that \begin{equation}\label{eq_nu_i}
                                                                                                                     \lim_{\nu\to 0} \nu_i^{p}\theta(\nu_i) =0
                                                                                                                                           \end{equation}
 (for if we had $\nu^k \theta(\nu) \ge  \frac{\eta}{\nu}$, for some $\eta>0$ and all $\nu>0$ small, then $\nu^k \theta(\nu)$ could not be $L^1$). Denote by $K$ the support of $\varphi$. We claim that \begin{equation}\label{eq_claim_omega_nu}
\lim_{i \to \infty} |\pi_{\nu_i}^*\omega_{|K}|_1=0.                                                                                                                                                                                    \end{equation}
Proving this claim will yield (\ref{eq_pi_nu_tend_vers_zero}).

Since $K$ is compact, there is a positive real number $s$ such that $K\subset h([s,1)\times N_{x_0})$.
By definition of $r$, for every $\nu \in (0,1)$ and $x\in K$ we have $\pi_\nu(x)=r_\mu(x)$, where $\mu=\frac{\nu}{|x-x_0|}$. Thanks to (\ref{item_H_bi}) of Theorem \ref{thm_local_conic_structure}, we deduce that $\pi_\nu$ is $C\nu$-Lipschitz on $K$ for some constant $C$ independent of $\nu\in (0,1)$. Hence, for $x\in K$ and $\nu \in (0,1) $ we have (since $j\ge 1$)
 \begin{equation}\label{eq_pi_nu_pullback}
   |\pi_\nu ^* \omega(x)|\lesssim  \nu^j|\omega(\pi_\nu(x))|\le \nu |\omega(\pi_\nu(x))| .
  \end{equation}

  We thus get for $\nu \in (0,s)$:
 $$|\pi^*_\nu \omega_{|K}|_1 \lesssim \int_{x\in K}\nu |\omega (\pi_\nu(x))|\le  \bp \int_{x\in K} \nu^p |\omega(\pi_\nu(x))|^p  \bq^{\frac{1}{p}} \hn^m (K)^{\frac{1}{q}},     $$ 
  by H\"older's inequality. Making the substitution $y=h^{-1}(x)$ in the last integral, this entails (since $h_\nu=\pi_\nu\circ h$) 
$$ |\pi^*_\nu \omega_{|K}|_1\lesssim \Big{(}\int_{y\in h^{-1}(K)}\nu^p |\omega (h_\nu(y))|^p\jac h(y)\bq^{\frac{1}{p}}\lesssim  \Big{(}\int_{y\in h^{-1}(K)}\nu^p |\omega (h_\nu(y))|^p\bq^{\frac{1}{p}}.$$

  We therefore can conclude that for $i \in \N$
$$ |\pi^*_{\nu_i} \omega_{|K}|_1 \lesssim  \Big{(}\int_s^1 \int_{z \in N_{x_0}} \nu_i^p|\omega (h_{\nu_i}(z))|^p\,dt\bq^{\frac{1}{p}}  =\bp(1-s) \nu_i^p\theta(\nu_i)\bq^{\jpp},$$
which tends to zero (by choice of the sequence $\nu_i$, see (\ref{eq_nu_i})).
This  establishes (\ref{eq_claim_omega_nu}),  which yields in turn (\ref{eq_pi_nu_tend_vers_zero}).
\end{proof}
\end{subsection}

\begin{subsection}{$L^p$ bounds}\label{sect_lp_bounds}

\begin{pro}\label{pro_K_0_borne_pr_L_p}
There is a constant $C$ such that for any large enough $p$  we have for each $L^p$ $j$-form $\omega$, $j\ge 1$, on $\ux$:
\begin{equation}\label{eq_K_0_borne_pr_Lp}
|\K_0 \omega|_p\le C |\omega|_p .\end{equation}
\end{pro}
\begin{proof}
%
Since $r_s$ is bi-Lipschitz for each $s>0$ (see section  \ref{sect_r_s} for $r_s$),  the function $(s,x)\mapsto \jac r_s(x)$ (defined on a subanalytic dense subset of $[0,1]\times \ux$) can only tend to zero when $s$ goes to zero. Consequently, by \L ojasiewicz's inequality (see Proposition \ref{pro_lojasiewicz_inequality}), there is a positive integer  $k$ and a constant $C$ such that  for almost all $(s,x)\in (0,1)\times \ux$  \begin{equation}\label{eq_jac_r_s}  s^k\le C \;\jac r_s(x) .  \end{equation}
 We shall establish (\ref{eq_K_0_borne_pr_Lp}) for all $p\in (k+1,\infty]$. 
 
 Let $p$ be a real number greater than $(k+1)$ (we postpone the case $p=\infty$) and let  $\omega$ be an $L^p$ $j$-form on $\ux$, $j\ge 1$.   We shall estimate $|\K_0 \omega|_p$  using (\ref{eq_K_0_2eme_forme}). 
 For this purpose, we first estimate the $L^p$ norm of $\omega\circ r_s$.  Indeed, setting $y=r_s(x)$, we see that
                  \begin{equation}\label{eq_estimate_omega_circ_r_s1}  |\omega\circ r_s|_p=\big{(}\int_{x\in \ux} |\omega(r_s(x))|^p \; \big{)}^{\frac{1}{p}} = \big{(}\int_{y\in r_s(\ux)} |\omega(y)|^p \cdot \jac r_s^{-1}(y)\; \big{)}^{\frac{1}{p}},\end{equation} 
which, by (\ref{eq_jac_r_s}), yields that
 \begin{equation}\label{eq_estimate_omega_circ_r_s2}
|\omega\circ r_s|_p   \le C^{\frac{1}{p}} \cdot s^{\frac{-k}{p}}\cdot|\omega|_p.
                   \end{equation}

Now, as $r_s^*$ has bounded derivative (by a constant independent of $s$) we have $|r_s^*\omega (x)|\le C'|\omega(r_s(x))|$, for some constant $C'$ independent of $x$ and $s$.
By (\ref{eq_K_0_2eme_forme}), we deduce \begin{equation}\label{eq_norme_K_0}|\K_0\omega(x)|\le C'\int_0^1 |\omega(r_s(x))|ds,\end{equation} which, thanks to Minkowski's inequality, entails that $$|\K_0\omega|_p\le C'\int_0^1 |\omega\circ r_s|_pds \overset{(\ref{eq_estimate_omega_circ_r_s2})}{\le} C^{\frac{1}{p}} C' |\omega|_p\int_0^1  s^{-\frac{k}{p}}ds ,$$
 showing that $\K_0$ is bounded for the $L^p$ norm independently of $p$ (since $p> k+1$).
 
 In the case $p=\infty$, it immediately follows from (\ref{eq_norme_K_0}) that $$|\K_0\omega|_\infty \le C'\int_0^1|\omega|_\infty\, ds= C'|\omega|_\infty,$$ for each  $L^\infty$ $j$-form $\omega$ on $\ux$. Hence, the result is clear in the case $p=\infty$ as well. \end{proof}

%

\begin{pro}\label{pro_K_1_borne_pr_L_p}
There is a constant $C$ such that for any  $p$ sufficiently close to $1$ we have for each $L^p$ form $\omega$ on $\ux$:
\begin{equation}\label{eq_K_1_borne_pr_Lp}
|\K_1 \omega|_p\le C |\omega|_p .\end{equation}
%
\end{pro}
\begin{proof}
%
As we can cover $M$ by finitely many orientable manifolds and use a partition of unity, we may assume that $\ux$ is oriented.  Take $\varphi \in \Lambda_0 ^{m-j+1}(\ux)$ as well as an $L^1$ $j$-form $\omega$ on $\ux$, and
observe that we have (for the relevant orientation on $N_{x_0}$)
\begin{eqnarray*}
 <\omega,\K_0\varphi>&=&\int_{\ux} \omega  \wedge\K_0 \varphi  \\&=&\int_{(0,1)\times N_{x_0}} h^*\omega\wedge  \hn_0 \varphi \qquad(\mbox{pulling back via $h$})\\
 &=&  \int_{x\in N_{x_0}} \int_0^1(h^*\omega)_{\pa_t}(t,x)\wedge   \hn_0 \varphi(t,x)dt  \\
 &=& \int_{x\in N_{x_0}} \int_0^1 \left(\int_0 ^t  (h^*\omega)_{\pa_t}(x,t)\wedge (h^*\varphi)_{\pa_s}(s,x)ds\right)dt\quad\mbox{(by (\ref{eq_h_0})})\\
  &=& \int_{x\in N_{x_0}} \int_{0<s\le t<1}  (h^*\omega)_{\pa_t}(t,x)\wedge(h^*\varphi)_{\pa_s}(s,x)ds\,dt.
\end{eqnarray*}
Making the same computation for $<\K_1 \omega ,\varphi>$ and applying Fubini's Theorem, we see that $$ <\omega,\K_0\varphi>=<\K_1\omega,\varphi>.$$ 

 Let now $q$ be a real number sufficiently big for the conclusion of Proposition \ref{pro_K_0_borne_pr_L_p} to hold (for every $L^q$ form) and denote by $p$ its H\"older conjugate  (which is close to $1$). By the above, for every $L^p$ $j$-form $\omega $ on $\ux$ and each  $\varphi \in \Lambda_0 ^{m-j+1}(\ux)$ we have:
$$|< \K_1 \omega,\varphi>|= <\omega,\K_0\varphi>\le |\K_0 \varphi|_q \cdot |\omega|_p \overset{(\ref{eq_K_0_borne_pr_Lp})}{\le}C |\varphi|_q  \cdot|\omega|_p, $$
which yields  (\ref{eq_K_1_borne_pr_Lp}).
\end{proof}
\end{subsection}
\end{section}

\begin{section}{Proof of the de Rham theorems}\label{sect_proof_de_rham_lp}
 Throughout this section, the letter  $M$ will stand for a bounded subanalytic submanifold of $\R^n$ and $X$ for its closure.

\begin{subsection}{The sheaves.}
 We will conclude by means of a sheaf theoretic argument. The problem is that $\Omega^j _{p}$ is not a sheaf on $M$. We thus shall work with the sheaf on $X$ of locally $L^p$ forms which has the same global sections (recall that $M$ is not compact).
 
 For $p\in [1,\infty)$ and $U\subset X$ open, let  $\f^j_{p}(U)$ be the $\R$-vector space of the $C^\infty$ $j$-forms $\omega$ on $U\cap M$ for which $\omega$ and $d\omega$ are both locally $L^p$ (locally in $U$, not in $U\cap M$),  i.e., those that satisfy for every $x_0\in U$ and  $\ep>0$ small enough 
$\int_{B(x_0,\ep)\cap M}|\omega|^p+|d\omega|^p<\infty$ (if $p<\infty$) 
or (in the case $p=\infty$)
                                        $\sup_{x\in B(x_0,\ep)\cap M} |\omega(x)|+|d\omega(x)|<\infty.$

 Clearly, $(\f^j_{p})_{j\in \N}$ is a complex of sheaves on $X$ for every $p\in [1,\infty]$. Observe that as $X$ is compact, locally $L^p$ is equivalent to $L^p$, which entails that $\F^j_{p}(X)=\op^j(M)$ for all $p \in [1,\infty]$. 
 Note also that all these sheaves are soft and therefore acyclic. 

 Given an open subset $U$ of $X$, we will denote by $\f^j_{p,c}(U)$ the sections of $\f^j_{p}(U)$ that are compactly supported.

Here, it is worthwhile stressing the fact that for $x_0\in \delta M$ the elements of $\f^j_{p,c}(B(x_0,\ep)\cap X)$ are forms on $B(x_0,\ep)\cap M $ which do not need  to be zero near the points of $\delta M$. 
Such forms just have to be zero near $S(x_0,\ep)$. In particular, they are not necessarily compactly supported as forms on $ B(x_0,\ep)\cap M$.
 
 Similarly, given an open subset $U$ of $X$, we will write $\fpb^j_p(U)$ for the space of weakly differentiable locally $L^p$ $j$-forms on $U\cap M$ that have an $L^p$ weak exterior differential, and $\fpb_{p,c}^j$ for the compactly supported sections of this sheaf.  Observe that the elements of  $\f^j_{p,c}(U)$ and $\fpb_{p,c}^j$ are $L^p$ forms on $U\cap M$. We have:
 \begin{lem}\label{lem_isom_compactly_supp_sections}
For all $p \in [1,\infty]$ and every open subset $U$ of $X$, the inclusions $\F^j_{p,c}(U)\hookrightarrow \fpb^{j}_{p,c}(U)$, $j\in \N$, induce isomorphisms in cohomology.
%
\end{lem}
 \begin{proof} If $\omega\in \fpb^{j}_{p,c}(U) $ then, by  Theorem \ref{thm_regularization_operators} (and Remark \ref{rem_regularization}), $R_i\omega\in\F^j_{p,c}(U) $ for all $i$ large enough. By (\ref{eq_ra}), this implies that  the mapping $R:H^{j}(\fpb^\bullet_{p,c}(U))\to H^{j}(\F^\bullet_{p,c}(U))$, defined by $R(\omega):=R_i(\omega)$, for $i$ large enough, is the inverse of the mapping induced by the inclusion between the two cochain complexes. 
 %
    \end{proof}
 
We also need to introduce a complex $\D^j_c$ of compactly supported singular oriented cochains in a similar way. Given an open subset $V$ of $M$ and $j \in \N$, let $C ^j(V)$ denote the cochain complex of the singular cochains of $V$.  It is a consequence of a well-known subdivision argument that although these presheaves are not sheaves on $M$, the respective  associated sheaves  $\C^j$ give rise to the same cohomology groups.  

Given now an open subset $U$ of $X$, we let $\D^j(U):=\C^j(U\cap M)$ and we will denote by $\D^j_c(U)$ the subspace of compactly supported sections.
 \end{subsection}

 \begin{subsection}{The case $p$ close to $1$}
The first step is to prove a Poincar\'e Lemma for $L^p$ forms with compact support. We show:

\begin{lem}\label{Poincare_lem_l1_simple}
  Given $x_0$ in $\delta M$,  there is $\ep>0$ such that for all $p\ge 1$ sufficiently close to $1$ and each  $j\in \N$,  we have:$$H^j(\F_{p,c}^\bullet(B(x_0,\ep)\cap X))\simeq 0.$$
   \end{lem}
\begin{proof}Let $\ep$ be some positive real number satisfying the conclusion of  Theorem \ref{thm_local_conic_structure} (which enables us to define the homotopy operators $\K_\nu$ of section \ref{sect_hom_hop}).
  A closed  $0$-form with compact support being  identically zero, the result is clear if $j=0$.  Let us thus fix a closed form $\omega  \in  \F_{p,c}^j ( B(x_0,\ep)\cap X )$ with $j>0$.  As $\omega$ is a  compactly supported section, by Proposition \ref{pro_proprietes_de_K_nu},   $\K_1\omega$ is a weakly differentiable $(j-1)$-form satisfying 
 $\overline{d}\K_1\omega=\omega$.  Furthermore, by Proposition  \ref{pro_K_1_borne_pr_L_p}, it  is $L^p$ if $p$ is sufficiently close to $1$.  By Lemma \ref{lem_isom_compactly_supp_sections}, this entails that   $\omega$ is the derivative of a compactly supported section on $B(x_0,\ep)\cap X$. 
\end{proof}

 The just above lemma holds {\it for $p$  sufficiently close to $1$} in the sense that there is $p_{0}\in (1,\infty]$ such that its statement holds for all $p \in [1, p_{0})$.  If we define $p_M(x_0)$ as the biggest such real number $p_0$, this number of course depends on the geometry of $X$ near $x_0$ and may  vary on this set. However, we have:
 \begin{lem}\label{lem_p_M_borne}
   $ \inf_{x_0 \in X} p_M(x_0)>1. $
\end{lem} 
  \begin{proof}If $M_{x_0}$ denotes the germ of $M$ at $x_0$, the family $(M_{x_0})_{x_0\in X}$ is  globally subanalytic.  As a matter of fact, by generic subanalytic bi-Lipschitz triviality (see Theorem 2.2 of \cite{vlt}), we know that there is a finite partition of $X$, such that given any two points $x_0$ and $x_0'$ in the same element of this partition, there is a (germ of) globally subanalytic bi-Lipschitz homeomorphism that maps  $M_{x_0}$ onto $M_{x_0'}$. Hence, by Proposition \ref{pro lp bi-lips invariant},  $p_M$ can take only  finitely many values.
  \end{proof}

 
 Given $k\in \N$ and an open subset $W$ of $M$, let %
\begin{eqnarray*}\phi_{W}^k:\Omega^k(W )&\to&
\C^{k}(W )\\
 \omega &\mapsto& [\phi_{W} ^j(\omega):\sigma \mapsto \int_\sigma
\omega].\end{eqnarray*}
\begin{thm}\label{thm_lq_psi_M_isom}
For each $p\ge 1$ sufficiently close to $1$ and each $j\in \N$, the mapping $\phi_{M} ^j$ induces an  isomorphism between $H_{p}^j(M) $ and  $H^j(M)$. 
\end{thm}
\begin{proof}
  Given an open subset $U$ of $X$ and $p\in [1,\infty)$, we denote by
   $\lambda_U^j:\F^j_{p,c} (U) \to \D^j_c(U) 
 $  the mapping induced by $\phi_{U\cap M}^j$. 
  It is easily checked from the definitions that  for all  $p\in [1,\infty)$
$$\E^j(U):= {\bf Hom} (\D_c^{m-j}(U),\R)\;\; \mbox{and}\;\; \mathcal{G}^j_p(U):= {\bf Hom} (\F^{m-j}_{p,c}(U),\R)$$ 
 are  complexes of flabby sheaves ($  \D^j_c$  and  $ \F^j_{p,c}$ are sometimes called {\it cosheaves} in the literature, see for instance \cite{bredon} Propositions V.1.6 and V.1.10). Moreover, the mappings  $\mu_U^{j}: \E^j(U)\to\mathcal{G}^j_p(U)$, $j\in \N$, defined as the respective adjoints of the $\lambda_U^{m-j}$, constitute a morphism of  complexes of sheaves.

It thus easily follows from sheaf theory (see for instance \cite{bredon}, section IV, Theorem 2.2) that it is  enough to show that for every $x_0\in X$ and every $\ep>0$ small enough, the mapping $\lambda^j_{B(x_0,\ep)\cap X}$ is an isomorphism for every $j$ (since the morphisms $\mu^j_{B(x_0,\ep)\cap X}$ are then isomorphisms as well).

  If $x_0$ is a point of $M$, this is a direct consequence of the usual Poincar\'e Lemma. We thus can assume that $x_0\in \delta M$, in which case, it easily comes down from the conic structure of $X$ at $x_0$ (see Theorem \ref{thm_local_conic_structure}) that 
%
%
  $$H^j(\D_c^\bullet(B(x_0,\ep)\cap X))\simeq 0,\quad  \mbox{for all $j$},$$
 so that the desired result follows from Lemma \ref{Poincare_lem_l1_simple} (for all $p$ close to $1$, see Lemma \ref{lem_p_M_borne}).
\end{proof}
\end{subsection}

 \begin{subsection}{The case where $p$ is large}
Fix $x_0\in X$ and set $\ux:=M \cap B(x_0,\ep)$, where $\ep>0$ is provided by  Theorem \ref{thm_local_conic_structure}.  We now have:

\begin{lem}\label{lem x normal Poincare lemma}(Poincar\'e Lemma for $p$ large) 
For $p\in[1,\infty]$ large enough, we have for all $j>0$: $$H_{p} ^j (\ux) \simeq 0.$$
\end{lem}
\begin{proof}
By Corollary \ref{cor_ws_isom_smooth}, it is enough to show that if $p$ is sufficiently large then  for every closed form $\omega \in \op^j(\ux)$, $j>0$, there is $\alpha \in \wsp^{j-1}(\ux)$ such that $\omega=\db \alpha$. 
But, by Propositions \ref{pro_dfn_K_0}, \ref{pro_proprietes_de_K_nu_Lp} and \ref{pro_K_0_borne_pr_L_p}, if $p$ is sufficiently large and if $\omega$ is such a form then   $\alpha:=\K_0\omega$ has all the required properties. 
\end{proof}

 We may here make an observation analogous to the one we made in Lemma \ref{lem_p_M_borne}. 
  If we define $q_M(x_0)$ as the smallest real number $q_0$ such that Lemma \ref{lem x normal Poincare lemma} holds for all $p \in (q_0,\infty]$, the same argument as in the proof of Lemma \ref{lem_p_M_borne} then establishes:
 \begin{equation}\label{eq_qm_borne}
   \sup_{x_0 \in X} q_M(x_0)<\infty. 
 \end{equation}

\begin{proof}[proof of Theorem \ref{thm_de_Rham_linfinity}]
In virtue of Theorem \ref{thm_intro_linfty}, it is enough to show that the inclusion of complexes $\F^j_{\infty} \to \F^j_{p}$ induces isomorphisms between the respective cohomology groups of the global sections. Since these sheaves are acyclic (these are fine sheaves), it suffices to prove that $(\F^j_{\infty})_{j\in \N}$  and $(\F^j_{p})_{j\in \N}$ constitute resolutions of the same sheaf $\A$  for each $p$ sufficiently large (see for instance \cite{bredon}, section II-4.2). But, if we define $\A(V)$ as the set of  $0$-forms $\omega:V\cap M\to \R$ which are locally constant (at every point of $V\cap M$), then, by Lemma \ref{lem x normal Poincare lemma}, the sequence
  $$0\hookrightarrow \A\hookrightarrow \F^0_{p} \overset{d}{\longrightarrow} \F^1_{p}  \overset{d}{\longrightarrow}\F^2_{p} \overset{d}{\longrightarrow} \dots $$
 is exact for all $p \ini$ sufficiently large (see (\ref{eq_qm_borne})).
\end{proof}
\end{subsection}

\begin{subsection}{An example.}
We end this paper by an example on which we discuss the
results of this paper. Let $X$ be the suspension of the torus.

%
%
%

 It is the set
constituted by two cones over a torus that are attached along this
torus. It is the most basic example on which Poincar\'e duality
fails for singular homology but holds for intersection homology \cite{ih1}.  Let $x_0$ and $x_1$ be the two
isolated singular points.

This set is a pseudomanifold which has very simple singularities
(metrically conical). However, the results of this paper show that
if they were not conical, this would not affect the
cohomology groups which only depend on the topology of the
underlying singular space.  This simple example is already enough to illustrate  how
the singularities affect Poincar\'e duality for $L^p$ cohomology.

Take $p$ in $[1,\infty]$ sufficiently close to $1$  for Theorem \ref{thm_de_Rham_lq}  to hold for $X_{reg}$. If $p$ is sufficiently close to $1$  then, by Theorem \ref{thm_de_Rham_linfinity}, we also have $ H_q ^j (X_{reg})\simeq I^\bt H^j (X)$, where $q$ is the H\"older conjugate of $p$. The  cohomology groups
involved in these theorems are gathered in the table below.

 \begin{center}
\begin{tabular}{|l|c|c|c|c|}
  \hline 

  {Cohomology groups}{$\qquad\qquad\quad j=$}
 & $0$& $1$& $2$ & $3$  \\
 \hline $ I^\bt H^j (X)$ and $ H_q ^j (X_{reg})$ & $\R$ & $0$  & $\R^2$  & $\R$     \\
  \hline $ I^0 H^j (X)$    & $\R$  &$\R^2$  & $0$   &$\R$   \\
  \hline  $H^j(X_{reg})$ and $H_{p}^j (X_{reg})$& $\R$   & $\R^2$  & $\R$  & $0$    \\
  \hline
\end{tabular}
\end{center}
\medskip

All these results may be obtained from the isomorphisms given in
section \ref{sect_framework} and a triangulation. Below, we interpret them geometrically. 

Let $T \subset X$ be the original torus  and let $\sigma$ and
$\tau$ be the two generators of $H_2(X)$  supported by the respective suspensions of the   two circles  generating the torus
$T$. For $\ep>0$ set: $$\sigma^\ep :=\{x \in |\sigma|: d(x,\{x_0,x_1\})
=\ep \} ,$$
where $|\sigma|$ stands for the support of the cycle $\sigma$.

If $\omega$ is an $L^q$ $2$-form that is equal to zero near the singular points
and that satisfies
\begin{equation}\label{eq_ex_sigma}\int_\sigma \omega=1,\end{equation} and if $\omega=d
\alpha$, for some $1$-form $\alpha$, then $\int_{\sigma^\ep} \alpha\equiv 1$ (by Stokes' formula). As
the volume of $\sigma^\ep$ tends to zero, $\alpha$ cannot be an 
$L^q$ form if $q$ is big. Consequently, if $\omega$ is an $L^q$ closed $2$-form,
zero near the singularities and satisfying (\ref{eq_ex_sigma}), it
must represent a nontrivial class.  This accounts for the fact that $ H_q ^2 (X_{reg})\simeq \R^2$.
In fact, every nontrivial class
may be represented by a shadow form \cite{bgm}.

 The nontrivial $L^p$ classes of $1$-forms are dual to the generators of the torus $T$ (while, as we have seen in the preceding paragraph, the nontrivial $L^q$ classes are dual to their respective suspensions $\sigma$ and $\tau$). We  see that $L^p$ cohomology is dual to $L^q$ cohomology 
in dimension $0$ and $1$ (as it is established by  Corollary
\ref{cor_2}).

However, the above form $\alpha$ may be $L^p$ and this accounts for the fact that the $L^p$ cohomology of the $2$-forms is not isomorphic to $\R^2$. The only nontrivial $L^p$ class of
$2$-forms is actually provided by the forms whose integral  on $T$ is
nonzero.   We see in particular that this collapsing torus induces a gap between  $H_{p}^{2} (X_{reg})$ and $H_{q} ^1
(X_{reg})$. 
 This is a typical example of the way the
singularities affect the duality between $L^p$ and $L^q$
cohomology.
\end{subsection}

\end{section}

\end{document}